\documentclass[11pt,a4paper]{amsart}
\usepackage[english]{babel}
\usepackage[utf8]{inputenc}  
\usepackage[T1]{fontenc}  
\usepackage{amscd,amssymb,amsmath,latexsym,url,mathrsfs,graphicx,upgreek}
\usepackage[hidelinks]{hyperref} \hypersetup{breaklinks=true}
\usepackage{latexsym}
\usepackage{amsmath}
\usepackage{amsthm}
\usepackage{amscd}
\usepackage{amssymb}
\usepackage{hyperref}
\usepackage{mdframed}
\usepackage{graphicx}
\usepackage{enumerate}
\usepackage{mathrsfs}
\usepackage{url}
\usepackage{tikz}
\usepackage{tikz-cd}
\usepackage{calc}
\usepackage{bm}
\usepackage{mathtools}

\newcommand{\abs}[1]{| #1 |}
\newcommand{\norm}[1]{\| #1 \|}
\newcommand{\Abs}[1]{\left\lvert#1\right\rvert}

\newcommand{\ol}[1]{\overline{#1}}

\usepackage{extarrows}%for x-long-arrows

\DeclareMathOperator{\Spec}{Spec}
\DeclareMathOperator{\vol}{vol}

\DeclareMathOperator{\ord}{ord}

\renewcommand{\and}{\quad \text{ and } \quad}

\newcommand{\codim}{{\operatorname{codim}}}

\newcommand{\h}{{\operatorname{h}}}

\renewcommand{\c}{{\operatorname{c}}}

\DeclareMathOperator{\e}{e}
\DeclareMathOperator{\K}{K}
\DeclareMathOperator{\an}{an}
\DeclareMathOperator{\tr}{tr}
\DeclareMathOperator{\Proj}{Proj}

\newcommand{\fM}{\mathfrak{M}}

\newcommand{\Gm}{\mathbb{G}_{\mathrm{m}}}

\newcommand{\Gal}{{\operatorname{Gal}}}

\let\div\relax
\DeclareMathOperator{\div}{div}

\newcommand{\NN}{{\mathbb{N}}}
\newcommand{\PP}{{\mathbb{P}}}
\newcommand{\QQ}{{\mathbb{Q}}}
\newcommand{\RR}{{\mathbb{R}}}

\newcommand{\ZZ}{{\mathbb{Z}}}

\def\fM {{\mathfrak M}}

\def\cO {{\mathcal O}}

\def\scrL {{\mathscr{L}}}
\def\scrM {{\mathscr{M}}}
\def\scrN {{\mathscr{N}}}

\def\scrO {{\mathscr{O}}}

\def\scrX {{\mathscr{X}}}
\def\scrY {{\mathscr{Y}}}

\newcommand{\bfzeta}{{\boldsymbol{\zeta}}}

\makeatletter
\def\subsection{\@startsection{subsection}{2}%
  \z@{.5\linespacing\@plus.7\linespacing}{-.5em}%
   { \bfseries}}
\makeatother

\newcounter{thm}

\theoremstyle{definition}
\newtheorem{definition}[thm]{Definition}

\newtheorem{remark}[thm]{Remark}
\newtheorem{example}[thm]{Example}
\newtheorem{question}[thm]{Question}

\newtheorem{assumption}[thm]{Assumption}

\theoremstyle{plain}
\newtheorem{lemma}[thm]{Lemma}
\newtheorem{proposition}[thm]{Proposition}
\newtheorem{theorem}[thm]{Theorem}
\newtheorem{corollary}[thm]{Corollary}

\newtheorem{prop-def}[thm]{Proposition-Definition}

\begin{document}

\title{Higher dimensional essential minima and equidistribution of cycles}
\author[Gualdi]{Roberto~Gualdi}
\address{Fakult\"at f\"ur Mathematik, Universit\"at Regensburg, 93040 Regensburg,Germany.}
\email{roberto.gualdi@ur.de}
\urladdr{\url{https://homepages.uni-regensburg.de/~gur23971/}}
\author[Mart\'inez]{C\'esar Mart\'inez}
\address{Fakult\"at f\"ur Mathematik, Universit\"at Regensburg, 93040 Regensburg,Germany.}
\email{cesar.martinez@ur.de}
\urladdr{\url{https://cesar-martinez-math.weebly.com}}

\date{\today}
\thanks{This work was supported by the SFB 1085 ``Higher Invariants'' funded by the DFG. The first author was also partially supported by the ``Fundació Ferran Sunyer i Balaguer'' and by the ``Alexander von Humboldt Foundation''}

\begin{abstract}
The essential minimum and equidistribution of small points are two well-established interrelated subjects in arithmetic geometry.
However, there is lack of an analogue of essential minimum dealing with higher dimensional subvarieties, and the equidistribution of these is a far less explored topic.

In this paper, we introduce a new notion of higher dimensional essential minimum and use it to prove equidistribution of generic and small effective cycles.
The latter generalizes the previous higher dimensional equidistribution theorems
by considering cycles and by allowing more flexibility on the arithmetic datum.
\end{abstract}
\maketitle

\setcounter{tocdepth}{1}
\tableofcontents

\vspace*{-8mm}

%
%
%
%
%
%
%
%
%
%
%
%
%
%
%
%
%
%
%
%INTRODUCTION
%
%
%
%
%
%
%
%
%
%
%
%
%
%
%
%
%
%
%

\section*{Introduction}

The well known Faltings's theorem is one of the most celebrated examples of the motto according to which the geometry of a variety governs its arithmetic.
Equidistribution phenomena represent an instance of the converse influence and show how certain arithmetic properties of a sequence of subvarieties prescribe their limit geometrical behaviour.
In addition to its intrinsic beauty, equidistribution theory has proven to be a key ingredient in classical diophantine problems;
most notably the proof of Bogomolov's conjecture in~\cite{Ullmo:Bogomolov} and~\cite{Zhang:Bogomolov}, which has inspired function field analogues and  generalizations for number fields (see for instance \cite{Yamaki:survey} and \cite{Chambert-Loir:Bogomolov-survey}, and the reference therein)

The first appearance of equidistribution in Arakelov geometry is due to Szpiro, Ullmo and Zhang in their cornerstone paper \cite{SUZ}.
This work inspired a lot of progress in the following years: using totally different techniques, Bilu proved in~\cite{Bilu:1997} an analogous theorem for strict sequences of points in tori, while Chambert-Loir \cite{Chambert-Loir:mesures-equidistribution}, and Favre and Rivera-Letelier~\cite{Favre-Rivera:2006} extended the equidistribution result to non-archimedean places.

The most general form of the equidistribution theorem in varieties defined over number fields was given by Yuan in~\cite{Yuan:2008}.
The analogous result for varieties defined over function fields was later obtained by Gubler in \cite{Gubler:2008}.
To state them, let $X$ be a projective variety defined over a field $K$ as aforementioned,
and $\ol{L}$ be an ample line bundle on $X$ equipped with a semipositive metric.
This choice allows to define a suitable (normalized) height function $\widehat{\h}_{\ol{L}}$ on the set of algebraic cycles of~$X_{\ol{K}}$,
making use of Arakelov geometry, see~Subsection~1.B.
Similarly, it associates to every subvariety $Y$ of $X_{\ol{K}}$ a measure
\[
\c_1(\ol{L}_v)^{\wedge\dim(Y)}\wedge\delta_{Y_v^{\an}}
\]
on the analytification of $X$ at a place $v$ of~$K$.
If $(x_m)_m$ is a generic sequence of points in $X_{\ol{K}}$ such that
\begin{equation}\label{equation:intro-convergence}
\widehat{\h}_{\ol{L}}(x_m)\longrightarrow\widehat{\h}_{\ol{L}}(X),
\end{equation}
the equidistribution theorem of Yuan-Gubler asserts the weak convergence of probability measures on $X_v^{\an}$
\[
\frac{1}{\#O(x_m)}
\sum_{y\in O(x_m)}\delta_{y_{v}^{\an}}
\longrightarrow
\frac{1}{\deg_L(X)}\,\c_1(\ol{L}_v)^{\wedge\dim(X)}\wedge\delta_{X_v^{\an}}
\]
for every place $v$, where $O(x_m)$ denotes the Galois orbit of $x_m$ in~$X_{\ol{K}}$.

However, the existence of generic sequences of algebraic points satisfying \eqref{equation:intro-convergence} fails for general choices of $X$ and~$\ol{L}$;
in these cases the equidistribution statement is empty.
A convenient invariant to deal with this issue is the \emph{essential minimum} of~$X$, that is defined as
\begin{equation}\label{equation:classical-essential-minimum}
\e_1(X,\ol{L}):=\adjustlimits\sup_{H}\inf_{x\not\in H}\widehat{\h}_{\ol{L}}(x),
\end{equation}
where $H$ runs over all closed subsets of $X_{\ol{K}}$ of codimension~$1$.
It is the smallest limit value that the height of a generic net of points in $X_{\ol{K}}$ can attain, and it can be shown that
\begin{equation}\label{equation:intro-Zhang-inequality}
\e_1(X,\ol{L})\geq \widehat{\h}_{\ol{L}}(X),
\end{equation}
see for instance \cite[Theorem~5.2]{Zhang:1995}.
This is known as \emph{Zhang's inequality} and plays a significant role in equidistribution theory.
For instance, the statement of Yuan-Gubler's theorem is nonempty precisely when~\eqref{equation:intro-Zhang-inequality} is an equality.

However, even under a strict inequality in~\eqref{equation:intro-Zhang-inequality}, it may happen that a sequence of generic points whose height converges to the essential minimum equidistributes with respect to a certain relevant measure.

A first example of this behaviour is when $X$ is a toric variety and $\ol{L}$ is a toric metrized line bundle.
An exhaustive description of this situation was given by Burgos Gil, Philippon, Rivera-Letelier and Sombra in~\cite{Burgos-Philippon-Rivera-Sombra:2015}.
They showed that equidistribution holds for a large class of toric metrized line bundles for which~\eqref{equation:intro-Zhang-inequality} is not necessarily an equality,
and explicitly described the limit measure.

A second relevant case is the one of a semiabelian variety $X$ defined over a number field as studied by K\"uhne~\cite{Kuhne:2018}.
In this case, the essential minimum of $X$ vanishes, whereas $\widehat{\h}_{\ol{L}}(X)$ can be negative (if $X$ is non-split).

\vspace{\baselineskip}

Let us now consider the higher dimensional situation.
In~\cite{Yuan:2008} (and also implicitely in~\cite{Gubler:2008}) an extension of the equidistribution theorem for subvarieties is stated,
generalizing a previous result by Autissier~\cite{Autissier:2006} (see also~\cite{Baker-Ih:equidistribution} for the special case of the N\'eron-Tate height on abelian varieties).
%As before, let $X$ be a projective variety defined over a field $K$ which is either a number field or the function field of a normal projective variety over a field, and let $\ol{L}$ be an ample line bundle on $X$ equipped with a semipositive metric.
Assume that the choice of the metric on $L$ satisfies the hypothesis
\begin{equation}\label{hypothesis *}
\widehat{\h}_{\ol{L}}(Y)\geq\widehat{\h}_{\ol{L}}(X)\ \text{ for every subvariety }Y \text{ of }X_{\overline{K}}.
\end{equation}
If $(Y_m)_m$ is a generic sequence of subvarieties of $X_{\overline{K}}$ of a fixed dimension such that
\begin{equation}\label{equation:intro-convergence-subvarieties}
\widehat{\h}_{\ol{L}}(Y_m)\longrightarrow\widehat{\h}_{\ol{L}}(X),
\end{equation}
the higher dimensional equidistribution theorem asserts that, for every place~$v$, the Galois average of the $v$-adic probability measures associated to $Y_m$ converges weakly to $\c_1(\ol{L})^{\wedge\dim(X)}\wedge\delta_{X_v^{\an}}/\deg_L(X)$.

The theorem relies on the fullfilment of \hyperref[hypothesis *]{hypothesis (\ref*{hypothesis *})}, which is an indispensable ingredient in the original proof of~\cite{Autissier:2006}.
Even if it holds in classical situations (such as the canonical height in toric varieties and the N\'eron-Tate height on abelian varieties), this assumption fails for general choices of $X$ and $\overline{L}$, see \hyperref[ex:failure of hypothesis *]{Example~\ref*{ex:failure of hypothesis *}} for an explicit construction.
Furthermore, as it happens for points, sequences of subvarieties satisfying \eqref{equation:intro-convergence-subvarieties} do not need to exist.
However, in contrast to the $0$-dimensional case, there is no appearance in the literature of a notion of essential minimum for higher dimensional subvarieties of~$X$.

The main goal in this paper is to deal with these two limitations.
In particular, give an equidistribution theorem that generalizes the one of Yuan to a situation where~\eqref{hypothesis *} is no longer needed,
and determine the cases in which this equidistribution theorem is nonempty by introducing a notion of higher dimensional essential minimum comparable to the classical one that suites this purpose.

\vspace{\baselineskip}

Let $K$ be a number field or the function field of a regular projective curve. Let $X$ be a projective variety defined over $K$ and $\ol{L}$ be a semipositive metrized line bundle $\ol{L}$ on $X$. We also fix $d=0,\ldots,\dim(X)$. 

We introduce in Definition~\ref{definition d-dimensional successive minima} a notion of higher essential minimum.
For simplicity in the introduction, let us assume that $L$ is ample.
In  this situation, \hyperref[prop:equivalent-def-ess-minimum]{Proposition~\ref*{prop:equivalent-def-ess-minimum}} allows the following equivalent definition.

\begin{definition}
The \emph{$d$-dimensional essential minimum of $X$ with respect to $\ol{L}$} is defined as
\[
\e_1^{(d)}(X,\ol{L})\,:=\;
\adjustlimits\sup_{H}
\inf_{Y\nsubseteq H}\,
\Bigg((d+1)\widehat{\h}_{\ol{L}}(Y)-
\inf_{\substack{s\in\mathrm{H}^0(X_{\ol{K}},L_{\ol{K}}^{\otimes n})\\ n\in\NN\setminus\lbrace 0\rbrace \\ Y\nsubseteq \abs{\div(s)}}}
d\,\widehat{\h}_{\ol{L}}(\div(s)\cdot Y)\Bigg).
\]
where $H$ runs over all closed subsets of $X_{\ol{K}}$ of codimension~$1$, and $Y$ over all $d$-dimensional subvarieties of~$X_{\ol{K}}$.
\end{definition}

The term in parenthesis in the above definition represents the highest gap between the height of $Y$ and the one of its divisors constructed from sections of powers of~$L$.
Then, by \hyperref[rmk:another equivalent definition]{Remark~\ref*{rmk:another equivalent definition}}, the $d$-dimensional essential minimum of $X$ can be seen as the minimal limit of such a highest ``height-gap'' for generic nets of $d$-dimensional subvarieties of~$X$.
When $d=0$, it agrees with the classical invariant defined in \eqref{equation:classical-essential-minimum}.

The explicit dependence of $\e_1^{(d)}(X,\ol{L})$ on the complete linear series of $L$ highlights the importance of understanding the arithmetic size of global sections of powers of~$L$.
In this perspective, we make use of the arguments of \cite{Yuan:2008} and \cite{Gubler:2008} to prove the existence of a global section of a tensor power of $L$ whose adelic norm is controlled by the height of~$X$. This allows us to deduce the following result, see \hyperref[cor:Zhangs-ineq]{Corollary \ref*{cor:Zhangs-ineq}}.

\begin{theorem}[Zhang's inequality]
We have~$\e_1^{(d)}(X,\ol{L})\geq\widehat{\h}_{\ol{L}}(X)$.
\end{theorem}

This theorem is the precise reason why we could not take a naive definition of higher essential minimum, which only involves the infimum value of (normalized) heights of generic subvarieties of a fixed dimension.
Indeed, the inequality of this theorem may fail with this alternative definition, see for instance \hyperref[ex:failure of hypothesis *]{Example~\ref*{ex:failure of Zhang ineq}}.

Having established the definition of $d$-dimensional essential minimum, we study its connection with equidistribution phenomena.
For this, let $(Y_m)_m$ be a net of $d$-dimensional subvarieties of $X_{\ol{K}}$.

\begin{definition}\label{def introduction generic and small}
The net $(Y_m)_m$ is said to be \emph{generic} if for every closed subset $H$ of $X_{\overline{K}}$ of codimension~$1$,  $Y_m\nsubseteq H$ for all $m$ big enough.
It is called \emph{$\ol{L}$-small} if
\[
\lim_m\,
\Bigg((d+1)\widehat{\h}_{\ol{L}}(Y_m)-\inf_{\substack{s\in\mathrm{H}^0(X_{\ol{K}},L_{\ol{K}}^{\otimes n})\\Y_m\nsubseteq |\div(s)|}}d\,\widehat{\h}_{\ol{L}}(\div(s)\cdot Y_m)\Bigg)
=\e_1^{(d)}(X,\ol{L}).
\]
\end{definition}

The notion of smallness is novel, as it is related to the higher dimensional essential minimum defined above.
Loosely speaking, generic $\ol{L}$-small nets of subvarieties are the ones for which the highest ``height-gap'' of their members has the smallest possible asymptotic behaviour.
With these concepts, we can predict the geometric behaviour of $d$-dimensional subvarieties as follows.

\begin{theorem}[equidistribution of subvarieties]\label{introthm:equidistribution}
Assume that $\e_1^{(d)}(X,\ol{L})=\widehat{\h}_{\ol{L}}(X)$.
If $(Y_m)_m$ is a generic and $\ol{L}$-small net of $d$-dimensional subvarieties of $X_{\ol{K}}$,
the weak convergence of probability measures on $X_v^{\an}$
\[
\frac{1}{\#O(Y_m)\deg_L(Y_m)}
\sum_{Y_m^{\sigma}\in O(Y_m)}\c_1(\ol{L}_v)^{\wedge d}\wedge\delta_{Y_{m,v}^{\sigma,\an}}
\longrightarrow
\frac{1}{\deg_L(X)}\,\c_1(\ol{L}_v)^{\wedge\dim(X)}\wedge\delta_{X_v^{\an}}
\]
holds for any place $v$, where $O(Y_m)$ denotes the set of Galois conjugates of $Y_m$ in~$X_{\ol{K}}$.
\end{theorem}

As in the case of points, the nonemptyness of this statement is ensured by the condition on the $d$-dimensional essential minimum, namely that Zhang's inequality is an equality. Moreover, the above convergence does not require the extra assumption~\eqref{hypothesis *}, and contains Yuan's equidistribution theorem in the higher dimensional situation (see Proposition~\ref{prop:autissier} and the comment that follows it).

In Theorem~\ref{equidistribution theorem} we prove a more general version of the equidistribution result, dealing with nets of effective cycles which are not necessarily Galois orbits.
In the case of points, a result of this kind can be deduced using a diagonal extraction argument, see for instance \cite[Corollary 8.6]{Dujardin} in the dynamical case.
%\cite{Gualdi-Sombra} (in the more general setting of logarithmic equidistribution).
Notice that our extension requires a generalization of the definitions of genericity and smallness for nets of effective cycles, which are given in \hyperref[definition genericity]{Definition \ref*{definition genericity}} and \hyperref[definition smallness]{Definition \ref*{definition smallness}} respectively.

\vspace{\baselineskip}

The point of view adopted in this paper opens interesting questions analogue to the case of points.
More precisely, it would be interesting to explore if, as in the case of points, the use of the higher dimensional essential minima has applications beyond Theorem~\ref{def introduction generic and small}.
For instance, can small and generic nets of $d$-dimensional cycles equidistribute even when~$\e_1^{(d)}(X,\ol{L})>\widehat{\h}_{\ol{L}}(X)$?
Under which conditions on the metrized line bundle does this happen?
Can the limit measure be described explicitly in these cases?
As for the $0$-dimensional situation, we hope that the toric and the semiabelian worlds may offer new insight and testing grounds for such questions.

\vspace{\baselineskip}

The paper is organized as follows.
In Section~1, we recall some preliminary material on arithmetic geometry and height theory.
In Section~2, we introduce the notion of higher essential minima and deduce their basic properties.
Section~3 is devoted to the proof of a key inequality (Theorem~\ref{thm:key-lemma}) for the equidistribution theorem, and from which we deduce our analogue of Zhang's inequality (Corollary~\ref{cor:Zhangs-ineq}).
In Section~4, we prove the equidistribution theorem in its general form (Theorem~\ref{equidistribution theorem}).
Finally, in Section~5, we first compare our result with the ones already present in literature and then explore its applications for heights arising from dynamical systems.

\vspace{\baselineskip}

\noindent\textbf{Acknowledgements.}
The authors would like to thank Walter Gubler and Mart\'{\i}n Sombra for many precious and fruitful discussions, as well as the anonymous referee for valuable remarks and suggestions.
They are also grateful to the Universities of Barcelona, Bordeaux, Caen and Regensburg for their hospitality while this research was carried on.

%
%
%
%
%
%
%
%
%
%
%
%
%
%
%
%
%
%
%
%TERMINOLOGY AND CONVENTION
%
%
%
%
%
%
%
%
%
%
%
%
%
%
%
%
%
%
%

\vspace{\baselineskip}

\begin{center}\textbf{Terminology and conventions}\end{center}
By a \emph{variety} over a field $K$ we mean a reduced and irreducible separated scheme of finite type over~$\Spec K$.
If $X$ is a variety over $K$, $L$ is a line bundle on $X$ and $K'$ a field extension of~$K$, we write $X_{K'}$ and $L_{K'}$ for the base change of $X$ and $L$ to $K'$.
If $X$ is a variety over a field, a \emph{subvariety} $Y$ of $X$ is a closed integral subscheme of~$X$.
We simply write~$Y\subseteq X$.
A \emph{$d$-cycle} or a \emph{cycle of pure dimension $d$} in $X$ is a formal finite sum of $d$-dimensional subvarieties of $X$.

For any place $v$ of a field $K$, we denote by $K_v$ the completion of $K$ with respect to the topology given by~$v$.
The algebraic closure of $K_v$ is equipped with a unique extension of $\abs{\cdot}_v$, and its completion with respect to such an absolute value is denoted by~$\mathbb{C}_v$.
If $X$ is variety over $K$ and $L$ a line bundle on~$X$, the notations $X_v^{\an}$ and $L_v^{\an}$ stand for the Berkovich analytifications of the base change of $X$ and $L$ over~$\mathbb{C}_v$.
The analytic space $X_v^{\an}$ comes with an action of the Galois group~$\Gal(\ol{K_v}/K_v)$.

%
%
%
%
%
%
%
%
%
%
%
%
%
%
%
%
%
%
%
%SECTION PRELIMINARIES
%
%
%
%
%
%
%
%
%
%
%
%
%
%
%
%
%
%
%

\numberwithin{equation}{section}
\numberwithin{thm}{section}
\section{Preliminaries in height theory}

We collect in this section the definitions and results in algebraic geometry and adelic Arakelov theory that are used throughout all the paper. In particular, we recall the usual adelic structure on number fields and function fields, the definition of local and global heights on varieties defined over such fields, and some geometrical and arithmetical notions of positivity of line bundles. We also introduce elementary perturbations of metrized line bundles and study their influence on heights of cycles.

\subsection{Number fields and function fields}

Throughout this paper, $K$ denotes either a number field or the function field of a regular projective curve defined over any field.
In both cases, $K$ can be given the structure of an 
\emph{adelic field}
in the sense of \cite[Definition 1.5.1]{Burgos-Philippon-Sombra:2014} by specifying a collection of places $\fM_K$ on~$K$, which we identify with a choice, for each~$v\in\fM_K$, of an absolute value $\abs{\cdot}_v$ on $K$ representing $v$ and a positive real weight~$n_v$.

\begin{definition}\label{def:places}
The adelic structure of the field $K$ is defined by the following choices.
\begin{enumerate}
\item If $K=\QQ$, the set $\fM_\QQ$ consists of the archimedean and the $p$-adic places of~$\QQ$, with corresponding absolute values normalized in the standard way, see \cite[\S1.2]{Bombieri-Gubler:2006}, and all weights equal to~$1$.
\item If $K=k(C)$, with $C$ a regular projective curve defined over a field $k$, the set $\fM_{k(C)}$ consists of all closed points of~$C$.
We associate to every $v\in\fM_{k(C)}$ the absolute value and weight given by
\[
\abs{\cdot}_v=c_k^{-\ord_v(\cdot)}
\quad\mbox{and}\quad
n_v=[k(v):k],
\]
where $\ord_v$ denotes the order of vanishing at $v$ and
\[
c_k:=
\begin{cases}
\#k &\mbox{if }\#k<\infty,
\\
\mathrm{e} & \mbox{otherwise.}
\end{cases}
\]
\item
If $K$ is a finite field extension of~$F$, where $F=\QQ$ or~$F=k(C)$ as in (2),
the set $\fM_K$ consists of all the places of $K$ which restrict to a place in $\fM_{F}$.
We associate to all $w\in\fM_K$ the unique absolute value $\abs{\cdot}_w$ on $K$ in $w$ restricting on $F$ to ~$\abs{\cdot}_v$ for some $v\in\fM_{F}$ and the weight
\[
n_w=\frac{\dim_{F_v}(E_w)}{[K:F]}n_v,
\]
where the $E_w$'s are the local Artinian $F_v$-algebra that appear in the decomposition of $K\otimes_{F}F_v$ and are in one-to-one correspondence with absolute values on $K$ over~$\abs{\cdot}_v$. We refer to \cite[Definition~3.5]{Martinez-Sombra:BKK} or \cite[Remark~2.5]{Gubler:Mfields} for more details about this construction. 
\end{enumerate}
\end{definition}

\begin{remark}
The definitions of the adelic structure of $K$ in (2) and (3) are compatible.
This means that if we have a finite map $C\rightarrow C'$, the adelic structure on $k(C)$ agrees with the one coming from~$k(C')$ by field extension.
On the other hand, any finite field extension of $k(C)$ can be seen as a function field of a finite cover of~$C$.
See~\cite[Example~3.9]{Martinez-Sombra:BKK} for details.
\end{remark}

Whenever it is clear from the context, the set of places of $K$ will be simply denoted by $\fM$. By construction, the adelic fields $(K,\fM)$ introduced in Definition~\ref{def:places} satisfy the \emph{product formula}, that is
\[
\sum_{v\in\fM}n_v\,\log\abs{\alpha}_v=0\quad\mbox{for every } \alpha\in\K^\times,
\]
see \cite[Proposition 1.4.4 and Proposition 1.4.7]{Bombieri-Gubler:2006}.

\subsection{Local and global heights}

Let $X$ be a projective variety over $K$, $L$ a line bundle on $X$, and $v\in\fM$ a fixed place of $K$.
A (\emph{continuous}) \emph{$v$-adic metric} on $L$ is the datum of a map $\norm{\cdot}_v\colon L_v^{\an}(U)\rightarrow Cont(U,\RR_{\geq 0})$ for each open subset~$U\subseteq X^{\an}_v$, satisfying the properties in \cite[\S1.1.1]{Chambert-Loir:2011}, with the additional requirement that it is invariant with respect to the action of $\Gal(\ol{K_v}/K_v)$ on~$X_v^{\an}$.
Whenever we want to stress the invariance under the given Galois group, we say that the metric is \emph{defined over~$K$}.
A line bundle $L$ with a continuous $v$-adic metric $\|\cdot\|_v$ is called a \emph{$v$-adic metrized line bundle} and it is denoted by $(L,\|\cdot\|_v)$ or, for short, by~$\overline{L}_v$. 

It is possible to define pull-backs, tensor products and inverses of $v$-adic metrized line bundles. This gives the set of isometry classes of $v$-adic metrized line bundles over $X$ the structure of an abelian group, see \cite[\S1.1.2]{Chambert-Loir:2011}; the neutral element is the class of $(\mathscr{O}_X,\|\cdot\|_{v,\tr})$, with $\|1\|_{v,\tr}=1$ defining the \emph{$v$-adic trivial metric} on $\mathscr{O}_X$.
If $\ol{L}_v=(L,\norm{\cdot}_v)$ is a $v$-adic metrized line bundle, we denote by $\norm{\cdot}_v^{\otimes n}$ the metric of~$\ol{L}_v^{\otimes n}$.
Also, if $(K',\fM^\prime)$ is an adelic finite field extension of $(K,\fM)$ and $w\in\fM^\prime$ is such that $w\mid v$, a continuous $v$-adic metric on $L$ defines a continuous $w$-adic metric on the extension of $L$ to $K^\prime$, as $\Gal(\ol{K^\prime_w}/K^\prime_w)\subseteq\Gal(\ol{K_v}/K_v)$.

For two continuous $v$-adic metrics $\|\cdot\|_{1,v}$ and $\|\cdot\|_{2,v}$ on $L$, their \emph{distance} is defined to be
\begin{equation}\label{definition distance of metrics}
\mathrm{d}_v(\norm{\cdot}_{1,v},\norm{\cdot}_{2,v}):=\sup_{p\in X_v^{\an}\setminus\div(s)}\big|\log (\norm{s(p)}_{1,v}/\norm{s(p)}_{2,v})\big|
\end{equation}
for any choice of a nonzero rational section $s$ of $L$.

When $v$ is an archimedean place of $K$, a continuous $v$-adic metric on $L^{\an}_v$ is said to be \emph{semipositive} if its associated first Chern current $\c_1(\ol{L}_v)$ is semipositive, see~\cite[\S1.2.8]{Chambert-Loir:2011} for more details.

When $v$ is a non-archimedean place of $K$, an \emph{algebraic} $v$-adic metric on $L$ is a metric $\norm{\cdot}_v$ on~$L^{\an}_v$ such that there is a nonzero $e\in\NN$ for which $\norm{\cdot}_v^{\otimes e}$ is induced by an algebraic $K_v^{\circ}$-model $(\mathscr{X},\mathscr{L})$ of $(X,L^{\otimes e})$ in the sense of~\cite[Definition~2.5 and Remark~2.6]{Gubler-Martin}.
Notice that this notion agrees with the one of formal metrics introduced in~\cite[\S7]{Gubler:localheights}, see~\cite[Proposition~8.13]{Gubler-Kuennemann:2017}.
The algebraic $v$-adic metric on $L$ induced by $\mathscr{L}$ is said to be \emph{semipositive} if $\mathscr{L}\cdot C\geq 0$ for every closed integral vertical curve $C$ in~$\mathscr{X}$. %in the literature, this property is sometimes called `vertically nefness'.
We refer to~\cite[Theorem~0.1]{Gubler-Kuennemann:positivity} for equivalent definitions of semipositivity of formal metrics using the language of forms and currents on Berkovich spaces.
More generally, a $v$-adic metric $\|\cdot\|_v$ on $L$ is said to be \emph{semipositive} if there exists a sequence of semipositive algebraic $v$-adic metrics on $L$ converging to $\|\cdot\|_v$ with respect to the distance defined in~\eqref{definition distance of metrics}.
For algebraic $v$-adic metrics, this agrees with the previous definition, see~\cite[Proposition~7.2]{Gubler-Kuennemann:positivity}.

Finally, a metrized line bundle $\overline{L}_v$ is said to be \emph{DSP}, short for difference of semipositive, if there exist semipositive metrized line bundles $\overline{M}_v$ and $\overline{N}_v$ such that
\[
\overline{L}_v\simeq\overline{M}_v\otimes\overline{N}_v^{-1}.
\]

For a $d$-dimensional subvariety $Y$ of~$X$ and the choice of a $d$-tuple of semipositive $v$-adic metrized line bundles $\overline{L}_{0,v},\ldots,\overline{L}_{d-1,v}$ on~$X$, one can construct a regular Borel measure
$\c_1(\overline{L}_{0,v})\wedge\ldots\wedge \c_1(\overline{L}_{d-1,v})\wedge\delta_{Y^{\an}_v}$
on $X_v^{\an}$, which is supported on $Y_v^{\an}$. In the archimedean case, it can be defined by Bedford-Taylor theory, see for instance \cite[Corollary 2.3]{Demailly:LelongNumbers}, while in the non-archimedean case it was first introduced in \cite[D\'efinition~2.4 and Proposition~2.7~b)]{Chambert-Loir:mesures-equidistribution} and later in~\cite[3.8]{Gubler:tropical} under relaxed assumptions.
Furthermore, this measure can be extended by multilinearity to a $d$-cycle $Z$  of~$X$, and we denote it by
\begin{equation}\label{eq: Chambert-Loir measure}
\c_1(\overline{L}_{0,v})\wedge\ldots\wedge \c_1(\overline{L}_{d-1,v})\wedge\delta_{Z^{\an}_v}.
\end{equation}
When all the $v$-adic metrized line bundles coincide with $\overline{L}_v$, one may simply write $\c_1(\overline{L}_v)^{\wedge d}\wedge\delta_{Z^{\an}_v}$.

%It will be referred to as the \emph{$v$-adic Chambert-Loir measure} associated to $\overline{L}_{1,v},\ldots,\overline{L}_{d,v}$ and $Y$.

\begin{proposition}\label{prop:measure-properties}
With the above hypotheses and notations, the measure in \eqref{eq: Chambert-Loir measure} satisfies the following properties:
\begin{enumerate}
\item it is a measure on $X_v^{\an}$ of total mass $\deg_{L_0,\ldots,L_{d-1}}(Z)$, and positive if $Z$ is effective;
\item it is symmetric and multilinear in the choice of $\overline{L}_{0,v},\ldots,\overline{L}_{d-1,v}$;
\item given, for all $i=0,\ldots,d-1$, a sequence $(\norm{\cdot}_{i,v,\ell})_\ell$ of continuous semipositive $v$-adic metrics on $L_i$ converging to $\norm{\cdot}_{i,v}$ with respect to the distance in~\eqref{definition distance of metrics}, then there is a weak convergence of the corresponding measures.
\end{enumerate}
\end{proposition}
\begin{proof}
We can assume that $Z$ is a prime cycle.
In the archimedean case, the claims are a consequence of the definition of the first Chern current and of the measure in \eqref{eq: Chambert-Loir measure}, of \cite[Proposition 1.2 and Corollary 1.10]{Demailly:LelongNumbers} and of the classical Wirtinger theorem.
In the non-archimedean case, these properties are proven in~\cite[Corollary~3.9~(a) and Proposition~3.12]{Gubler:tropical}.
\end{proof}

The measure in~\eqref{eq: Chambert-Loir measure} allows the definition of the \emph{local height} of a $d$-cycle $Z$ of $X$ with respect to the choice of pairs $(\ol{L}_{i,v},s_i)$, $i=0,\ldots,d$, consisting of a semipositive $v$-adic metrized line bundle on $X$ and a rational section $s_i$ of $L_i$,
such that $s_0,\ldots,s_d$ intersect $Z$ properly.
We define $\h(\emptyset):=0$ and, for $d\geq 0$, we follow the recursive formula
\begin{multline}\label{definition of local height}
\h_{(\ol{L}_{0,v},s_0),\ldots,(\ol{L}_{d,v},s_d)}(Z):=
\h_{(\ol{L}_{0,v},s_0),\ldots,(\ol{L}_{d-1,v},s_{d-1})}(\div(s_d)\cdot Z)
\\
-\int_{X_v^{\an}}\log\norm{s_d}_{d,v}\ \c_1(\ol{L}_{0,v})\wedge\cdots\wedge \c_1(\ol{L}_{d-1,v})\wedge\delta_{Z_v^{\an}}.
\end{multline}
It is symmetric and multilinear in the choice of $\ol{L}_{0,v},\ldots,\ol{L}_{d,v}$.
Moreover, we can extend this definition to DSP $v$-adic metrized line bundles.

\begin{remark}\label{rmk:continuity-of-height}
Let $Z$ be a $d$-cycle, and $(L_0,s_0),\ldots,(L_d,s_d)$ fixed line bundles on $X$ equipped with rational sections intersecting $Z$ properly. By~\eqref{definition of local height} and \hyperref[prop:measure-properties]{Proposition~\ref*{prop:measure-properties}}, the function
\[
(\norm{\cdot}_0,\ldots,\norm{\cdot}_d)\longmapsto\h_{((L_0,\norm{\cdot}_0),s_0),\ldots,((L_d,\norm{\cdot}_d),s_d)}(Z)
\]
is Lipschitz continuous on the set of $(d+1)$-tuples of DSP $v$-adic metrics on $L_0,\dots,L_d$ respectively.
%\[
%\|\h_{((L_0,\norm{\cdot}_0),s_0),\ldots,((L_d,\norm{\cdot}_d),s_d)}(Y) - \h_{((L_0,\norm{\cdot}'_0),s_0),\ldots,((L_d,\norm{\cdot}'_d),s_d)}(Y)\|
%\leq  \sum_i\deg_{\hat{L_i}}(Y)d_v(\norm{\cdot}_i,\norm{\cdot}'_i)
%\]
\end{remark}

Next, we deal with the adelic structure of $(K,\fM)$.
For this we combine the local pieces of information introduced above with some coherence condition.
A \emph{metrized line bundle} $\ol{L}:=(L,(\norm{\cdot}_v))$ is a line bundle $L$ together with a $v$-adic metric for each place~$v\in\fM$. It is called \emph{semipositive} (respectively \emph{DSP}) if the $v$-adic metric $\|\cdot\|_v$ is semipositive (respectively DSP) for all $v\in\mathfrak{M}$.

A metrized line bundle $\overline{L}$ is said to be \emph{quasi-algebraic} if there exists a finite set $S\subseteq\fM$ containing all archimedean places, a nonzero $e\in\NN$ and an algebraic $K^{\circ}_S$-model $(\scrX,\scrL)$ of $(X,L^{\otimes e})$, such that for each $v\not\in S$ the metric $\norm{\cdot}^{\otimes e}_v$ is induced by localizing the model at~$v$.
A quasi-algebraic metrized line bundle is called~\emph{algebraic} if $S$ coincides precisely with the set of archimedean places.
Pull-backs, tensor products and inverses of quasi-algebraic line bundles are again such.

\begin{proposition}\label{prop:panico}
Let $\ol{L}=(L,(\|\cdot\|_v))$ be a quasi-algebraic metrized line bundle on~$X$, defined over~$K$ and such that $\|\cdot\|_v$ is an algebraic $v$-adic metric on $L$ for all non-archimedean places $v$ of~$K$.
Then, there exists a finite field extension $K'$ of $K$ such that the base change of $\ol{L}$ to $K'$ is algebraic.
\end{proposition}
\begin{proof}
By considering a suitable positive integer tensor power of $L$, we can assume that every non-archimedean $v$-adic metric is given by a $K^\circ_v$-model of $(X,L)$.
In such a case, this follows from \cite[Lemma~3.5]{Yuan:2008} and \cite[Proposition~3.4]{Gubler:2008}.
Notice that the equivalence between $v$-adic formal and algebraic metrics is proven in~\cite[Theorem~1.1]{Gubler-Martin}.
\end{proof}

Recall that in our setting, saying that $\ol{L}$ is defined over~$K$ is solely involved with the $\Gal(\ol{K}_v/K_v)$-invariance of the $v$-adic metrics.

\begin{remark}
Given a DSP quasi-algebraic metrized line bundle  $\ol{L}=(L,(\norm{\cdot}_v))$,
for every closed point $p$ of $X$ and every rational section $s$ of $L$ such that $p\not\in |\div(s)|$,
we have
\[
\norm{s(p^{\an}_v)}_v=1,
\]
for almost every place $v\in\fM$.
This is enough to show that, given a $d$-cycle $Z$ of $X$ and a family of DSP quasi-algebraic metrized line bundles together with rational sections $(\ol{L}_i,s_i)$, $i=0,\ldots,d$, such that $s_0,\ldots,s_d$ intersect $Z$ properly, 
\[
\h_{((L_0,\norm{\cdot}_{0,v}),s_0),\ldots,((L_d,\norm{\cdot}_{d,v}),s_d)}(Z)=0,
\]
for almost every~$v\in\fM$.
We refer to~\cite[Propositions~1.5.8 and~1.5.14]{Burgos-Philippon-Sombra:2014} for more details about these statements.
\end{remark}

The previous remark allows to define global heights as finite weighted sums of local heights.
More precisely, given a family of DSP quasi-algebraic metrized line bundles $\ol{L}_0,\ldots,\ol{L}_d$, the \emph{global height} of a $d$-cycle $Z$ of $X$ is set to be
\[
\h_{\ol{L}_0,\ldots,\ol{L}_d}(Z):=\sum_{v\in\fM} n_v\, \h_{((L_0,\norm{\cdot}_{0,v}),s_0),\ldots,((L_d,\norm{\cdot}_{d,v}),s_d)}(Z),
\]
where $s_i$ is a rational section of $L_i$, $i=0,\ldots,d$, such that $s_0,\ldots,s_d$ intersect $Z$ properly.
Whenever all metrized line bundles coincide, we write $\h_{\ol{L}}(Z)$ for short.

This definition does not depend on the choice of the sections, by combining~\cite[Corollary~3.8]{Gubler:Mfields} and the fact that the product formula holds.
Moreover, it is symmetric and multilinear in the choice of $\ol{L}_0,\ldots,\ol{L}_d$, and invariant by finite field extensions.
In particular, one obtains a well-defined height function on the $d$-dimensional cycles of $X$ over~$\ol{K}$ by considering any finite field extensions of $K$ over which a cycle is defined and equipping it with the the structure given in \hyperref[def:places]{Definition~\ref{def:places}}.

The height can be seen as the arithmetic analogue of the notion of the \emph{degree} of a cycle with respect to a line bundle~$L$.
Indeed, recall that for a closed point $p$ of $X$ one sets~$\deg_L(p):=[K(p):K]$, which extends by linearity to $0$-cycles of~$X$.
Furthermore, for any $d$-cycle $Z$ of~$X$, its degree is defined recursively by $\deg_L(Z):=\deg_L(\div(s)\cdot Z)$ for an arbitrary rational section $s$ of $L$ intersecting $Z$ properly.

\begin{remark}\label{rmk:height=degree}
When $K$ is the function field of a curve and $\ol{L}$ is a semipositive algebraic metrized line bundle with metrics given by an algebraic model $(\scrX,\scrL)$ of $(X,L^{\otimes e})$, the height $\h_{\overline{L}}(Y)$ of a subvariety $Y$ of $X$ equals
$\deg_{\scrL}(\scrY)$,
where $\scrY$ is the closure of $Y$ in~$\scrX$.
\end{remark}

If $Y$ is a  subvariety of $X$ with~$\deg_L(Y)\neq0$, we set the \emph{normalized height of $Y$ with respect to a DSP quasi-algebraic metrized line bundle $\ol{L}$} to be
\begin{equation}\label{normalized height}
\widehat{\h}_{\ol{L}}(Y):=\frac{\h_{\ol{L}}(Y)}{(\dim(Y)+1)\deg_{L}(Y)}.
\end{equation}

\begin{example}
When~$X=\mathbb{P}_{\mathbb{Q}}^n$, there exists a choice of a canonical metric on the line bundle $\mathscr{O}(1)$ for which the associated height agrees with the classical one introduced by Weil and Northcott for algebraic numbers, see~\cite[Examples 1.3.11 and~1.4.4]{Burgos-Philippon-Sombra:2014} for the precise definition.
In such a case, for instance, the closed point $p$ of $\mathbb{P}^1_{\mathbb{Q}}=\Proj\mathbb{Q}[t_0,t_1]$ given by the homogeneous polynomial~$t_1^k-2t_0^k$ satisfies
\[
\h_{\ol{\mathscr{O}(1)}}(p)=\log 2 \quad\text{and}\quad \widehat{\h}_{\ol{\mathscr{O}(1)}}(p)=\frac{\log 2}{k},
\]
for all~$k\geq1$.
\end{example}

For any two collections $(\norm{\cdot}_{1,v})$ and $(\norm{\cdot}_{2,v})$ of $v$-adic metrics on $L$ such that $\|\cdot\|_{1,v}=\|\cdot\|_{2,v}$ for all but finitely many $v\in\fM$, we define their \emph{distance} as
\[
\mathrm{d}\big((\norm{\cdot}_{1,v}),(\norm{\cdot}_{2,v})\big)
:=
\sum_{v\in\fM}n_v\mathrm{d}_v(\norm{\cdot}_{1,v},\norm{\cdot}_{2,v}).
\]
The $v$-adic data of two quasi-algebraic metrics on a line bundle~$L$ coincide in all but a finite number of places. %We take always a common model dominating the two of them---> these don't change the metric.

\begin{lemma}\label{lemma:Lipschitz}
For every fixed $d$-cycle $Z$ of $X$, the function
\[
((\norm{\cdot}_{0,v}),\ldots,(\norm{\cdot}_{d,v}))\longmapsto\h_{(L_0,(\norm{\cdot}_{0,v})),\ldots,(L_d,(\norm{\cdot}_{d,v}))}(Z)
\]
is Lipschitz continuous on the set of $(d+1)$-tuples of DSP quasi-algebraic collections of $v$-adic metrics on $L_0,\dots,L_d$ respectively.
\end{lemma}

\begin{proof}
This follows from \hyperref[rmk:continuity-of-height]{Remark~\ref*{rmk:continuity-of-height}} and the multilinearity of heights with respect to DSP metrized line bundles.
\end{proof}

\subsection{Elementary perturbations of metrized line bundles}
Let $X$ be a projective variety over~$K$.
We here introduce a relevant class of continuous functions on~$X_v^{\an}$, for some $v\in\fM$, that play a central role in the proof of the equidistribution theorem in Section~4.

For the convenience of presentation, we unify two well-known archimedean and non-archimedean notions under a common name.

\begin{definition}
Let $v\in\mathfrak{M}$.
A real-valued function $f$ on $X_v^{\an}$ is said to be a \emph{$v$-adic elementary function} if
\begin{enumerate}
\item when $v$ is archimedean, $f$ is smooth;
\item when $v$ is non archimedean, $f$ is piecewise $\mathbb{Q}$-linear in the sense of~\cite[Definition~2.11]{Gubler-Martin}.
\end{enumerate}
\end{definition}

The following proposition relates $v$-adic elementary functions with metrics on $\mathscr{O}_X$ defined on a suitable finite field extension of $K$.

\begin{proposition}\label{prop:elementary functions and metrized trivial line bundle}
Let $v\in\mathfrak{M}$, and let $f$ be a $v$-adic elementary function on $X$.
There exists a finite field extension $K^\prime$ of $K$, and a place $w$ of $K^\prime$ over $v$ such that the choice $-\log\|1\|:=f$ determines a $w$-adic metric on $\mathscr{O}_X$ which is defined over $K^\prime$.
Moreover, when $v$ is non-archimedean, the so-defined $w$-adic metric is algebraic.
\end{proposition}
\begin{proof}
First, notice that $v$-adic elementary functions are continuous on $X_v^{\an}$, see \cite[Proposition~2.12(a)]{Gubler-Martin} for the non-archimedean case.
Then, the first claim follows from the fact that there exists a finite field extension $K^\prime$ of $K$, and a place $w$ of $K^\prime$, with~$w\mid v$, for which the function $f$ is invariant under the action of $\Gal(\ol{K^\prime_w}/K^\prime_w)$ on $X_v^{\an}$.
This statement is clear in the archimedean case.
If $v$ is non-archimedean, \cite[Proposition 2.18(b)]{Gubler-Martin} implies that there exists a finite field extension $F$ of $K_v$ such that $f$ is $\Gal(\ol{K_v}/F)$-invariant.
Since the extension $F/(F\cap K_v^{\mathrm{sep}})$ is purely inseparable, such a group of automorphisms coincides with~$\Gal(\ol{K_v}/F\cap K_v^{\mathrm{sep}})$, and the conclusion follows from~\cite[Exercise~2 page~30]{Serre:local-fields}.

For the last claim, notice that the metric is piecewise $\mathbb{Q}$-linear in the sense of~\cite[Definition~2.11]{Gubler-Martin}.
This is equivalent, using the compactness of $X_v^{\an}$, to the existence of a positive integer $e$ such that $\|\cdot\|^e$ is a piecewise linear metric on $L^{\otimes e}$ according to \cite[Definition~2.8]{Gubler-Martin}.
It suffices to apply \cite[Propositions~8.11 and~8.13]{Gubler-Kuennemann:2017} (see also \cite[Theorem~1.1]{Gubler-Martin}) to conclude that $\|\cdot\|^{\otimes e}$ is induced by an algebraic $(K_w^\prime)^\circ$-model of $L_{K^\prime}^{\otimes e}$.
\end{proof}

\begin{remark}
The proof of the previous proposition shows that non-archimedean elementary functions coincide with the model functions of \cite[Definition 3.4]{Yuan:2008}.
\end{remark}

Elementary functions are dense in the set $\mathcal{C}(X_v^{\an},\mathbb{R})$ of real-valued continuous functions on $X_v^{\an}$.

\begin{theorem}\label{density of brick functions}
Let $v\in\fM$. The set of $v$-adic elementary functions is a $\mathbb{Q}$-vector subspace of $\mathcal{C}(X_v^{\an},\mathbb{R})$.
Moreover, it is dense in $\mathcal{C}(X_v^{\an},\mathbb{R})$ with respect to the uniform convergence topology.
\end{theorem}
\begin{proof}
When $v$ is non-archimedean, the sum of two $v$-adic elementary functions is again such because of \cite[Proposition~2.12(b)]{Gubler-Martin}.
The other properties can be checked directly from the definitions.

For the second claim, both the archimedean and non-archimedean case are proved using the Stone-Weierstrass theorem.
In particular, the non-achimedean situation is shown in \cite[Theorem~7.12]{Gubler:localheights}, see also \cite[Proposition~2.15]{Gubler-Martin}.
\end{proof}

One of the main techniques in the proof of the equidistribution theorem is to consider slight perturbations of a given metrized line bundle by means of analytic functions.

\begin{definition}\label{perturbed line bundle}
Let $\overline{L}=(L,(\norm{\cdot}_v))$ be a metrized line bundle over $X$, $v_0\in\mathfrak{M}$, $f$ a continuous real-valued $\Gal(\ol{K}_{v_0}/K_{v_0})$-invariant function on $X_{v_0}^{\an}$, and $t\in\mathbb{Q}$.
The \emph{$(v_0,f,t)$-perturbation} of $\overline{L}$ is the line bundle $L$ together with the metric defined by
\begin{equation*}
\|s \|'_{v}:=
\begin{cases}
\norm{s}_v\;e^{-tf}&\text{if }v=v_0,\\
\norm{s}_v&\text{otherwise}
\end{cases}
\end{equation*}
for all local section $s$ of~$L$.
We denote this metrized line bundle by~$\overline{L}(v_0,f,t)$.
\end{definition}

\begin{remark}\label{rmk:elementary functions give perturbation after finite field extension}
Let $\ol{L}$ be a metrized line bundle over $X$, $v_0\in\mathfrak{M}$ and $f$ a $v_0$-adic elementary function.
It follows from the proof of \hyperref[prop:elementary functions and metrized trivial line bundle]{Proposition \ref*{prop:elementary functions and metrized trivial line bundle}} that $f$ is $\Gal(\ol{K^\prime}_{w_0}/K^\prime_{w_0})$-invariant for a certain finite field extension $K^\prime$ of $K$ and~$w_0\mid v_0$.
Then, it determines a $(w_0,f,t)$-perturbation of $\ol{L}$ defined over $K^\prime$ for all~$t\in\mathbb{Q}$.
\end{remark}

Perturbations via elementary functions satisfy the following favorable property.

\begin{lemma}\label{perturbation of DSP qa is DSP qa after a finite field extension}
Let $v_0\in\mathfrak{M}$, $f$ a $v_0$-adic elementary function, and $t\in\QQ$.
Let also $K^\prime$ and $w_0$ be as in \hyperref[rmk:elementary functions give perturbation after finite field extension]{Remark~\ref*{rmk:elementary functions give perturbation after finite field extension}}.
If $\overline{L}$ is a DSP quasi-algebraic metrized line bundle on~$X$, then $\overline{L}(w_0,f, t)$ is a DSP quasi-algebraic metrized line bundle on~$X_{K^\prime}$.
\end{lemma}
\begin{proof}
Finite base changes of quasi-algebraic metrized line bundles are again such.
Hence, the quasi-algebricity of $\overline{L}(w_0,f, t)$ follows from the fact that its metric coincides with the one of the extension of $\ol{L}$ to $K^\prime$ for all except one place.

Denote by $\ol{N}$ the line bundle $\mathscr{O}_X$ defined over $K^\prime$, equipped with the $w_0$-adic metric satisfying $-\log\|1\|_{w_0}=tf$, and the trivial metric at all other places of $K^\prime$.
Since
\[
\ol{L}(w_0,f,t)=\ol{L}_{K^\prime}\otimes\ol{N}
\]
and the tensor product of DSP metrized line bundles is again such, we can restrict to prove that $\ol{N}$ is DSP.

The trivial metric is semipositive; this follows from definition at archimedean places, and from~\cite[Proposition~6.4(b)]{Gubler-Kuennemann:positivity} otherwise.
So it is left to show that $\norm{\cdot}_{w_0}$ is DSP.

For the case when $w_0$ is non-archimedean, by \hyperref[prop:elementary functions and metrized trivial line bundle]{Proposition~\ref*{prop:elementary functions and metrized trivial line bundle}} and the closure of elementary functions under rational multiplication, there exists $e\in\mathbb{N}$ for which $\ol{N}^{\otimes e}$ comes from an algebraic line bundle $\scrN$ on a model $\scrX$ over~$(K_{w_0}^\prime)^\circ$.
Then, writing $\scrN$ as a difference of two ample line bundles on $\scrX$ gives that $\ol{N}^{\otimes e}$ is DSP, then the result.
When $w_0$ is archimedean, the smoothness of $f$ implies that $\ol{N}_{w_0}$ is DSP by tensoring by sufficiently positive metrized line bundle.
\end{proof}

The next lemma concerns the variation of the height of a cycle under this kind of perturbations.

\begin{lemma}\label{lemma about perturbed height}
Let $\overline{L}$ be a semipositive quasi-algebraic metrized line bundle over~$X$.
Let $v_0\in\mathfrak{M}$, $f$ a $v_0$-adic elementary function, $t\in\QQ$, and $K^\prime$ and $w_0$ as in \hyperref[rmk:elementary functions give perturbation after finite field extension]{Remark~\ref*{rmk:elementary functions give perturbation after finite field extension}}.
For every $d$-cycle $Z$ of~$X$,
\[
\h_{\overline{L}(w_0,f,t)}(Z)=
\h_{\overline{L}}(Z)+t\,(d+1)\,n_{w_0}\int_{X_{v_0}^{\an}}f\ \c_1(\overline{L}_{v_0})^{\wedge d}\wedge\delta_{Z^{\an}_{v_0}}+t^2P(t),
\]
with $P$ a polynomial with real coefficients and degree $d-1$, depending on~$\ol{L}$, $Z$ and~$f$.
\end{lemma}
\begin{proof}
Denote by $\ol{N}$ the $(w_0,f,1)$-perturbation of the trivial metrized line bundle, %$(\scrO_X,(\norm{\cdot}_{v,\mathrm{tr}}))$
so that $\ol{L}(w_0,f,t)=\ol{L}\otimes \ol{N}^{\otimes t}$.
By \hyperref[perturbation of DSP qa is DSP qa after a finite field extension]{Lemma~\ref*{perturbation of DSP qa is DSP qa after a finite field extension}}, $\ol{N}$ is DSP and quasi-algebraic on~$X_{K^\prime}$, hence we can write $\ol{N}\simeq \ol{M}_1\otimes \ol{M}_2^{-1}$, where $\ol{M}_1$ and $\ol{M}_2$ are semipositive quasi-algebraic metrized line bundles defined over~$K^\prime$.

Let $s$ be a rational section of $M_1$ that intersects $Z$ properly; %(which exists by Chow's moving lemma)
it is also a rational section of $M_2$.
Then, by multilinearity on $\ol{N}$ and the inductive definition of height,
\begin{equation*}
\begin{split}
\h_{\ol{L},\ldots,\ol{L},\ol{N}}(Z)
&=\h_{\ol{L},\ldots,\ol{L},\ol{M}_1}(Z)-\h_{\ol{L},\ldots,\ol{L},\ol{M}_2}(Z)\\
&=-\sum_{w\in\fM_{K^\prime}}n_{w} \int_{X_{w}^{\an}}\log\frac{\norm{s}_{M_1,w}}{\norm{s}_{M_2,w}}\;\c_1(\ol{L}_w)^{\wedge d}\wedge\delta_{Z_{w}^{\an}}\\
&=- n_{w_0}\int_{X_{w_0}^{\an}}\log\norm{1}_{\ol{N},w_0}\;\c_1(\ol{L}_{w_0})^{\wedge d}\wedge\delta_{Z_{w_0}^{\an}}\\
&=n_{w_0}\int_{X_{w_0}^{\an}}f\ \c_1(\overline{L}_{w_0})^{\wedge d}\wedge\delta_{Z^{\an}_{w_0}}.
\end{split}
\end{equation*}

Since the global height of $d$-cycles is symmetric and multilinear in the metrized line bundles, the above equality yields
\begin{equation*}
\begin{split}
\h_{\overline{L}(w_0,f,t)}(Z)
&=\h_{\overline{L}\otimes\ol{N}^{\otimes t}}(Z)
=\sum_{\ell=0}^{d+1}\binom{d+1}{\ell}\;\h_{\overline{L},\dots,\overline{L},\underbrace{\scriptstyle\overline{N}^{\otimes t},\dots,\overline{N}^{\otimes t}}_\ell}(Z)\\
&=\sum_{\ell=0}^{d+1}\binom{d+1}{\ell}\;t^{\ell}\;\h_{\overline{L},\dots,\overline{L},\underbrace{\scriptstyle\overline{N},\dots,\overline{N}}_\ell}(Z)\\
&=\h_{\overline{L}}(Z)+t\,(d+1)\,n_{w_0}\int_{X_{w_0}^{\an}}f\, \c_1(\overline{L}_{w_0})^{\wedge d}\wedge\delta_{Z_{w_0}^{\an}}+t^2P(t).
\end{split}
\end{equation*}
To conclude the proof, notice that since $w_0\mid v_0$ and both $L$ and $Z$ are defined over~$K$, the integral coincides with the one over~$X_{v_0}^{\an}$.
\end{proof}

\subsection{Positivity in arithmetic geometry}

Let $X$ be a projective variety over a field~$K$. In this subsection, we recall different notions of positivity in algebraic geometry and their arithmetic counterparts.

A line bundle $L$ on $X$ is said to be \emph{nef} if $\deg_L(C)\geq 0$ for every curve $C$ in~$X$. By Kleiman's theorem, see \cite[Theorem 1.4.9]{Lazarsfeld:positivityI}, this is equivalent to the fact that $\deg_L(Y)\geq0$ for all subvarieties $Y$ of $X$. Proper pull-backs, tensor products and positive powers of nef line bundles are again nef.
A line bundle $L$ is said to be \emph{semiample} if $L^{\otimes n}$ is globally generated for some $n>0$.
Notice that semiample line bundles are nef.

The \emph{volume} of a line bundle $L$ is defined as the nonnegative real number
\[
\vol(L):=\limsup_{n\rightarrow \infty}\frac{\dim\mathrm{H}^0(X,L^{\otimes n})}{n^{\dim(X)}/\dim(X)!};
\]
and $L$ is said to be \emph{big} if $\vol(L)>0$.
If $L$ is nef, \cite[Corollary 1.4.41]{Lazarsfeld:positivityI} asserts that $\vol(L)=\deg_L(X)$. In particular, if $L$ is big and nef, then $\deg_L(X)>0$.
\\Finally, big and nef line bundles rejoice the useful property that the degree of generic subvarieties is strictly positive, in the sense of the following proposition.

\begin{proposition}\label{prop:generic-big}
Let $X$ be a projective variety over a field $K$ and $L$ a big and nef line bundle on $X$. Then, there exists a Zariski closed subset $H_0\subseteq X_{\overline{K}}$ of codimension~$1$,
such that for every subvariety $Y$ of $X_{\overline{K}}$ that is not contained in $H_0$ one has~$\deg_L(Y)>0$.
\end{proposition}
\begin{proof}
As $L$ is defined over $K$ and the degree is invariant under base field extension and Galois action, the restriction of $L_{\overline{K}}$ to each irreducible component of $X_{\overline{K}}$ is again big and nef.
Then, the statement follows from the fact that the generic restriction of a big line bundle on an irreducible projective variety is again big, see~\cite[Corollary~2.2.11]{Lazarsfeld:positivityI}.
\end{proof}

\begin{remark}
Notice that when $L$ is ample, the strict positivity of the degree holds for every subvariety $Y$ of $X$ by the Nakai-Moishezon criterion, see~\cite[Theorem~1.2.23]{Lazarsfeld:positivityI}.
\end{remark}

There exist analogous notions for nefness and bigness  in the arithmetic case.
Let $\ol{L}$ be a semipositive quasi-algebraic metrized line bundle on~$X$.
We say that $\ol{L}$ is \emph{(arithmetically) nef} %This is also called horizontally semipositive in literature. (e.g. zhang)
if $L$ is nef and $\h_{\ol{L}}(p)\geq  0$ for every closed point $p$ in~$X$.

A global section $s\in \mathrm{H}^0(X,L)$ is called \emph{$\ol{L}$-small} if
\begin{equation}
\log\sup\norm{s}_v\leq 0\quad\mbox{for every }v\in\fM.
\end{equation}
In the context of Arakelov geometry, such sections are the arithmetic analogue of global sections in the geometric case.
This analogy is strengthened by the following observation, which characterizes the small sections of an algebraically metrized line bundle on a variety defined over a function field.

\begin{remark}\label{remark:small sections over function fields}
When $K$ is a function field, and $\ol{L}$ is endowed with the algebraic metric coming from a model $\mathscr{L}$, $\ol{L}$-small sections are identified with global sections of $\mathscr{L}$, see the proof of \cite[Proposition 2.2]{Chambert-Loir-Thuillier:equidistribution-logarithmique}.
\end{remark}

Moriwaki introduced in~\cite{Moriwaki:Continuity} the notion of \emph{arithmetic volume} of a quasi-algebraic metrized line bundle~$\ol{L}$, which is defined as
\[
\widehat{\vol}(\ol{L}):=\limsup_{n\rightarrow \infty}\frac{\log\#\big\{s\in\mathrm{H}^0\big(X,{L}^{\otimes n}\big)\,\vert\; s\mbox{ is $\ol{L}^{\otimes n}$-small}\big\}}{n^{\dim(X)+1}/(\dim(X)+1)!}
\]
in analogy with the geometric situation.

This is especially useful when $K$ is a number field, where small sections are more delicate to control than in the setting of Remark~\ref{remark:small sections over function fields}.
In this case, assuming that $\ol{L}$ is algebraically metrized, Chen  showed that the arithmetic volume is in fact a limit, see~\cite[Theorem 5.2]{Chen:Positive-degree}.

\begin{remark}\label{arithmetic volume over function fields}
Let $K$ be the function field of a smooth projective curve defined over a finite field $k$,
and $\ol{L}$ an algebraically metrized line bundle on $X$ defined by the model $(\mathscr{X},\mathscr{L})$.
Then $\widehat{\vol}(\ol{L})=\log(\#k) \vol(\mathscr{L})$.%with $\mathscr{L}$ considered over $k$.
\end{remark}

The relevance of the existence of small sections for some integer power of a metrized line bundle leads to the definition of arithmetic bigness.
Following \cite[Definition~2.1]{Yuan:2008}, a quasi-algebraic metrized line bundle $\ol{L}$ is said to be \emph{(arithmetically) big} if $\widehat{\vol}(\ol{L})>0$.
The reader is referred to \cite[\S2]{Moriwaki:Finitely-generated}, \cite[\S2.2]{Yuan:2008} and \cite{Moriwaki:Continuity} for equivalent definitions and properties regarding this notion.

\begin{remark}
If $K$ is a number field, and $\ol{L}$ is a big algebraic metrized line bundle on $X$, then $L$ is geometrically big; see~\cite[Proposition~5.1]{Chen:Positive-degree}.
\end{remark}

The following is a consequence to the generalized Hodge index theorem of Moriwaki, see~\cite[Corollary~6.14]{Moriwaki:Arakelov}, together with the continuity property of the arithmetic volume given by \cite[Theorem~5.1]{Moriwaki:Continuity2} and of the height function proven in Lemma~\ref{lemma:Lipschitz}.

\begin{lemma}\label{lemma:Moriwaki}
Let $K$ be a number field, and $\ol{L}$ be a semipositive algebraic  metrized line bundle on~$X$.
Then
\[
\widehat{\vol}(\ol{L})\geq \h_{\ol{L}}(X).
\]
In particular, if $\h_{\ol{L}}(X)>0$, then $\ol{L}$ is big.
\end{lemma}

Analogously, one proves the following extension of the translation made by Moriwaki in~\cite[Remark~6.5]{Moriwaki:Arakelov} of Yuan's version of Siu's theorem~\cite{Yuan:2008}.

\begin{lemma}\label{lemma:Siu-aritmetico}
Let $K$ be a number field, and $\ol{L}_1$, $\ol{L}_2$ be nef semipositive algebraic metrized line bundles on~$X$.
Then
\[
\widehat{\vol}(\ol{L}_1\otimes\ol{L}_2^{-1})\geq \h_{\ol{L}_1}(X)-(\dim(X)+1)\,\h_{\ol{L}_1,\ldots,\ol{L}_1,\ol{L}_2}(X).
\]
\end{lemma}

%
%
%
%
%
%
%
%
%
%
%
%
%
%
%
%
%
%
%
%SECTION ESSENTIAL MINIMA
%
%
%
%
%
%
%
%
%
%
%
%
%
%
%
%
%
%
%

\section{Higher dimensional essential minima}

Let $X$ be a projective variety over~$K$, and $\ol{L}=(L,(\norm{\cdot}_v))$ a semipositive quasi-algebraic  metrized line bundle on~$X$. In this section we define the successive minima of $X$ with respect to $\ol{L}$ for arbitrary dimensional subvarieties and we prove some basic properties of them, mainly focusing on the essential minimum. We work analytically with algebraically closed complete fields, reason for which we refer to the treatment of~\cite{Gubler-Hertel}.

\subsection{Correcting integrals}

Let $Y$ be a  subvariety of $X_{\ol{K}}$ of dimension $d$ and $s$ a nonzero rational section of an integer power $L_{\ol{K}}^{\otimes n}$ of the line bundle~$L$, satisfying $Y\nsubseteq|\div(s)|$.
After choosing a finite field extension $K^\prime$ of $K$ over which $Y$ and $s$ are defined, we set, for every~$w\in\fM_{K^\prime}$,
\begin{equation}\label{definition of local integrals}
I_{\ol{L}_w}(Y,s):=\frac{1}{n}\int_{X_w^{\an}}-\log\|s\|_w^{\otimes n}\ \c_1(\ol{L}_w)^{\wedge d}\wedge\delta_{Y^{\an}_w}.
\end{equation}
It is a well defined real number because of~\cite[Theorem~I]{Gubler-Hertel}.
Moreover, it is invariant under the tensor powering application $s\mapsto s^{\otimes m}$ from the set of rational sections of $L_{K^\prime}^{\otimes n}$ to the ones of $L_{K^\prime}^{\otimes nm}$, for all $n,m\in\mathbb{N}_{>0}$.
As $\overline{L}$ is quasi-algebraic, \cite[Theorem~3.1.13]{Gubler-Hertel} ensures that the function $w\mapsto |I_{\ol{L}_w}(Y,s)|$ is summable on $\mathfrak{M}_{K^\prime}$, so one can define the real number
\begin{equation}\label{definition of sums of integrals}
I_{\ol{L}}(Y,s):=\sum_{w\in\mathfrak{M}_{K^\prime}}n_w\,I_{\ol{L}_w}(Y,s).
\end{equation}

\begin{remark}\label{correcting integrals as height jumps}
It follows from the global induction formula of~\cite[Theorem~3.1.13]{Gubler-Hertel} and the multilinearity of the height with respect to the choice of metrized line bundles that
\[
I_{\overline{L}}(Y,s)=\h_{\ol{L}}(Y)-\frac{\h_{\ol{L}}(\div(s)\cdot Y)}{n},
\]
for every rational section $s$ of $L_{\ol{K}}^{\otimes n}$ with $Y\nsubseteq|\div(s)|$.
In particular, $I_{\ol{L}}(Y,s)$ is independent on the choice of~$K^\prime$.
\end{remark}

Consider a $d$-cycle $Z$ of~$X_{\ol{K}}$ and a section $s$ of $L_{\ol{K}}^{\otimes n}$ intersecting $Z$ properly, that is, no summand of the base change of $Z$ to $\ol{K}$ is contained in $|\div(s)|$.
Then, \eqref{definition of sums of integrals} extends linearly to define~$I_{\ol{L}}(Z,s)$.

We can readily compute the influence of the perturbation of a metric as in Subsection~1.C on the correcting integrals.

\begin{lemma}\label{correcting integrals for perturbed metrics}
Let $\ol{L}$ be a semipositive quasi-algebraic metrized line bundle on~$X$, and $Z$ a $d$-cycle of~$X$.
Let $v_0\in\mathfrak{M}$, $f$ a $v_0$-adic elementary function, $t\in\QQ$, and $K^\prime$ and $w_0$ as in \hyperref[rmk:elementary functions give perturbation after finite field extension]{Remark~\ref*{rmk:elementary functions give perturbation after finite field extension}}.
%Denote by $\ol{L}(w_0,f,t)=(L,(\norm{\cdot}'_w))$ the corresponding perturbed metrized line bundle defined on~$K^\prime$.
Then,
\[
I_{\ol{L}}(Z,s)
=
I_{\ol{L}(w_0,f,t)}(Z,s)
%\frac{1}{n}\sum_{w\in\mathfrak{M}_{K^\prime}}n_w\int_{X_w^{\an}}-\log\|s\|_{w}^{\prime\otimes n}\ \c_1(\ol{L}_w)^{\wedge d}\wedge\delta_{Z_w^{\an}}
-n_{w_0}\, t\,\int_{X_{v_0}^{\an}}f\ \c_1(\ol{L}_{v_0})^{\wedge d}\wedge\delta_{Z_{v_0}^{\an}}
\]
for every rational section $s$ of $L_{\ol{K}}^{\otimes n}$ intersecting~$Z$ properly.
\end{lemma}
\begin{proof}
The fact that both $Z$ and $\ol{L}$ are defined over $K$ ensures that the integrals of $f$ over $X_{w_0}^{\an}$ and $X_{v_0}^{\an}$ coincides.
With this observation, the claim follows by the definition of the metric of $\ol{L}(w_0,f,t)$ and by comparing correcting integrals over a finite extension of $K^\prime$ over which $s$ is defined.
\end{proof}

\subsection{Successive minima}
The goal of this subsection is to give a tool to control the values of the correcting integrals in~\eqref{definition of local integrals} for generic subvarieties of $X$ of a fixed dimension.
To do so, consider first, for any $d=0,\ldots,\dim(X)$ and $\eta\in\RR$, the closed subset of $X_{\ol{K}}$
\[
X^{(d)}(\eta,\ol{L}):=\ol{\bigcup Y}^{\mathrm{Zar}},
\]
where the union ranges over all $d$-dimensional subvarieties $Y$ of $X_{\ol{K}}$ that satisfy the inequality
\begin{equation}\label{eq:essentialminimum-ineq}
\sup\Big\{I_{\ol{L}}(Y,s)\, \mid\, s\in\mathrm{H}^0\big(X_{\ol{K}},L_{\ol{K}}^{\otimes n}\big), n\in\mathbb{N}\setminus\{0\},Y\nsubseteq|\div(s)|\Big\}\leq\eta\deg_L(Y),
\end{equation}
with the convention that the supremum of the empty set is~$-\infty$.

\begin{remark}\label{equivalence of condition for inequality defining minima}
As a consequence to \hyperref[correcting integrals as height jumps]{Remark \ref*{correcting integrals as height jumps}} we can rewrite condition~\eqref{eq:essentialminimum-ineq} as
\[
\h_{\ol{L}}(Y)-\inf_{\substack{s\in\mathrm{H}^0(X_{\ol{K}},L_{\ol{K}}^{\otimes n})\\n\in\NN\setminus\{0\}\\ Y\nsubseteq\abs{\div(s)}}}
\frac{\h_{\ol{L}}(\div(s)\cdot Y)}{n}
\leq\eta\deg_L(Y),
\]
expressing it in terms of ``height-gaps'' with respect to its Cartier divisors.
\end{remark}

\begin{definition}\label{definition d-dimensional successive minima}
For $d=0,\ldots,\dim(X)$ and $j=1,\ldots,\dim(X)+1-d$, we define the \emph{$j$-th $d$-dimensional successive minimum} of $X$ with respect to $\ol{L}$ as
\[
\e_j^{(d)}(X,\ol{L}):=\inf\big\{\eta\in\RR\,\mid\; \dim\big(X^{(d)}(\eta,\ol{L})\big)\geq \dim(X)+1-j\big\}\in\mathbb{R}\cup\{\pm\infty\}.
\]

The first $d$-dimensional successive minimum is referred to as the \emph{$d$-dimensional essential minimum} of $X$ with respect to~$\ol{L}$.
In particular, when $X$ is geometrically irreducible one has
\[
\e_1^{(d)}(X,\ol{L})=\inf\lbrace\eta\in\RR\,\mid\; X^{(d)}(\eta,\ol{L})=X_{\ol{K}}\rbrace.
\]
\end{definition}

Roughly speaking, the $d$-dimensional essential minimum encodes the generic highest jump that can be realized in the first step of the inductive definition of the height of $d$-dimensional subvarieties. 

\begin{remark}\label{rmk:ess-min-0}
Since by definition $\h_{\ol{L}}(\emptyset)=0$ and $\deg_L(p)=1$ for every point $p$ of~$X_{\overline{K}}$, the set $X^{(0)}(\eta,\ol{L})$ is the Zariski closure in $X_{\overline{K}}$ of the set of points whose height is upper bounded by~$\eta$.
In particular, for each~$j=1,\ldots,\dim(X)+1$, the invariant $\e_j^{(0)}(X,\overline{L})$ coincides with the classical notion of $j$-th successive minimum of $X$ with respect to~$\ol{L}$, see for instance~\cite[\S5]{Zhang:1995}.
\end{remark}

\begin{example}
Let $\ol{L}=(\mathscr{O}_X,(\norm{\cdot}_{v,\mathrm{tr}}))$.
The degree of any subvariety $Y$ of dimension at least $1$ with respect to $\mathscr{O}_X$ is zero, as well as the quantity $I_{\ol{L}}(Y,s)$ for all nonzero global section $s$ of $\mathscr{O}_X$, because of the product formula.
Hence
\[
\e_j^{(d)}(X,\ol{L})=-\infty,
\]
for all $d=1,\dots,\dim(X) $ and $j=1,\dots,\dim(X)+1-d$.
\end{example}

The situation is better behaved under some geometrical assumptions on the line bundle $L$. Recall that $L$ is said to have infinite Iitaka dimension if $\mathrm{H}^0(X,L^{\otimes n})=\{0\}$ for all $n>0$; see \cite[\S2.1]{Lazarsfeld:positivityI} for a more general definition and treatment in the case of normal varieties.

\begin{lemma}\label{first properties of successive minima}
Let  $d=0,\dots,\dim(X)$. Then
\begin{enumerate}
\item with the usual order relation on $\mathbb{R}\cup\{\pm\infty\}$,
\[\e_{\dim(X)+1-d}^{(d)}(X,\ol{L})\leq\ldots\leq\e_2^{(d)}(X,\ol{L})\leq\e_1^{(d)}(X,\ol{L});\]
\item if $L$ has infinite Itaka dimension, then all the $d$-dimensional successive minima equal $-\infty$;
\item if $L$ is nef, for all $\eta_1,\eta_2\in\mathbb{R}$,\[\eta_1\leq\eta_2\ \Longrightarrow\ X^{(d)}(\eta_1,\ol{L})\subseteq X^{(d)}(\eta_2,\ol{L});\]
\item if $L$ is big and nef, the $d$-dimensional essential minimum differs from $-\infty$.
\end{enumerate}
\end{lemma}
\begin{proof}
The first and third statements follow directly from the definition.
To prove~(2), notice that if $L$ has infinite Iitaka dimension, condition~\eqref{eq:essentialminimum-ineq} is satisfied for all subvarieties $Y$ of $X_{\ol{K}}$ and for all $\eta\in\mathbb{R}$.

Finally, assume that $L$ is big and nef.
Let $H_0$ be the Zariski closed subset of $X_{\ol{K}}$ given by~\hyperref[prop:generic-big]{Proposition~\ref*{prop:generic-big}}.
By bigness, there exists a nonzero global section $s$ of $L_{\ol{K}}^{\otimes n}$ for some $n>0$.
For every $v\in\mathfrak{M}$, the compactness of $X_v^{\an}$ and the continuity of the metric imply that $\|s\|^{\otimes n}_v$ is upper bounded by a strictly positive real constant $C_v$ on $X_v^{\an}$. Moreover, one can take $C_v=1$ for almost all $v\in\mathfrak{M}$ because of the quasi-algebricity of the metric.
Write $C:=-\sum_{v\in\mathfrak{M}}n_v\log C_v\in\mathbb{R}$ and set
\[\eta_s:=-\frac{|C|}{n}-1.\]
Notice that if $Y$ is a $d$-dimensional subvariety of $X_{\ol{K}}$ satisfying condition~\eqref{eq:essentialminimum-ineq} for such~$\eta_s$, then $Y$ has to be contained in $H_0\cup|\div(s)|$.
Indeed, if not, one would have
\[\eta_s\deg_L(Y)\geq I_{\ol{L}}(Y,s)\geq \frac{C}{n}\deg_L(Y)>\eta_s\deg_L(Y),\]
which is a contradiction.
Then,
\[X^{(d)}(\eta_s,\ol{L})\subseteq H_0\cup|\div(s)|.\]
Together with point (3), this implies that $\dim(X^{(d)}(\eta,\ol{L}))\leq\dim(X)-1$ for all~$\eta\leq\eta_s$, which in turn yields~$\e_1^{(d)}(X,\ol{L})\geq\eta_s$, concluding the proof.
\end{proof}

\begin{remark}
It is easy to adapt the proof of the fourth bullet of \hyperref[first properties of successive minima]{Lemma \ref*{first properties of successive minima}} to prove that $-|C|/n$ is in fact a lower bound for the $d$-dimensional essential minimum. More generally, the idea of controlling the size of global sections of $L$ to deduce lower bounds on the $d$-dimensional essential minimum is a central strategy in this paper and is exploited in the next section to relate such an arithmetic invariant with the (normalized) height of the ambient variety.
\end{remark}

Under the assumption that $L$ is semiample, a stronger conclusion than the one of \hyperref[first properties of successive minima]{Lemma~\ref*{first properties of successive minima}(4)} can be obtained.
We phrase it allowing fexibility in the choice of the base field, as follows.

\begin{lemma}\label{lemma about lower absolute constant}
Let $K^\prime$ be an algebraic extension of $K$. If $L$ is semiample, there exists an absolute real constant $C_{\ol{L},K^\prime}$ such that
\[
\sup_{\substack{s\in\mathrm{H}^0(X_{\ol{K}},L_{\ol{K}}^{\otimes n})\\n\in\NN\setminus\{0\}\\s\text{\emph{ intersects }}Y\text{\emph{ properly}}}}I_{\ol{L}}(Y,s)
\geq C_{\ol{L},K^\prime}\,\deg_L(Y)
\]
for every subvariety $Y$ of~$X_{K^\prime}$.
\end{lemma}
\begin{proof}
The base change $L_{K^\prime}$ is semiample. Then, there exists a finite family of nonzero global sections $s_0,\dots,s_r$ of $L_{K^\prime}^{\otimes n}$, for some~$n>0$,
such that for every subvariety $Y$ of~$X_{K^\prime}$, there is an $i=0,\dots,r$ for which $Y\nsubseteq|\div(s_i)|$.
By continuity of the metric of~$\ol{L}$, the compactness of the analytifications of $X$ and the quasi-algebricity of~$\ol{L}$, the quantity
\[
C_{\ol{L},K^\prime}:=\min_{i = 0,\dots,r}\Bigg(-\sum_{w\in\mathfrak{M}_{K^\prime}}n_w\frac{1}{n}\log\sup \|s_i\|_w\Bigg)
\]
is a real number.
Then, \hyperref[prop:measure-properties]{Proposition \ref*{prop:measure-properties} (1)} gives
\[
\sup_{\substack{s\in H^0(X_{\ol{K}},L_{\ol{K}}^{\otimes n})\\n\in\mathbb{N}\setminus\{0\}\\s\text{ intersects }Y\text{ properly}}}I_{\ol{L}}(Y,s)\geq C_{\ol{L},K^\prime}\deg_L(Y),
\]
concluding the proof.
\end{proof}

For the remaining of the paper, we focus on the notion of $d$-dimensional essential minimum and its applications to equidistribution theory.
An explicit form of this invariant in the case when $d=0$ is described in Remark~\ref{rmk:ess-min-0};
the other extreme case, $d=\dim(X)$, is illustrated in the following example.

\begin{example}\label{essential minimum of full dimension}
Assume that $X$ is geometrically irreducible.
If $L$ is big and nef, $\deg_L(X)>0$ and thus
\[
\e_1^{(\dim(X))}(X,\ol{L})=\sup\bigg\{\frac{I_{\ol{L}}(X,s)}{\deg_L(X)}\ \Big\vert \ s\in\mathrm{H}^0(X_{\ol{K}},L_{\ol{K}}^{\otimes n})\setminus\{0\},\ n\in\mathbb{N}\setminus\{0\}\bigg\}.
\]
Using \hyperref[correcting integrals as height jumps]{Remark \ref*{correcting integrals as height jumps}},
such an arithmetic invariant is completely determined by the knowledge of the height of $X$ and of all $1$-codimensional subvarieties of~$X_{\ol{K}}$. 
\end{example}

This example generalizes to the following alternative definition of the $d$-dimensional essential minimum.

\begin{proposition}\label{prop:equivalent-def-ess-minimum}
Let $L$ be a big and nef line bundle on $X$ and let $H_0$ denote the Zariski closed subset of $X_{\ol{K}}$ given by \hyperref[prop:generic-big]{Proposition \ref*{prop:generic-big}}.
For every $d=0,\ldots,\dim(X)$, we have
\begin{equation}\label{eq:essentialminimum}
\e_1^{(d)}(X,\ol{L})=
\sup_{\substack{H\text{\emph{ closed subset of }}X_{\ol{K}}\\\codim(H)=1}}\;
\inf_{\substack{Y\subseteq X_{\ol{K}}\\\dim(Y)=d\\ Y\nsubseteq H\cup H_0}}\;
\sup_{\substack{s\in\mathrm{H}^0(X_{\ol{K}},L_{\ol{K}}^{\otimes n})\\n\in\NN\setminus\{0\}\\Y\nsubseteq \abs{\div(s)}}}
\frac{I_{\ol{L}}(Y,s)}{\deg_{L}(Y)}.
\end{equation}
\end{proposition}
\begin{proof}
Fix $d=0,\ldots,\dim(X)$ and, to simplify the notation, let $\widetilde{\eta}$ denote the quantity on the right hand side of~\eqref{eq:essentialminimum}.
Fix an arbitrary choice of $\varepsilon>0$. For all closed subset $H$ of $X_{\ol{K}}$ of codimension~$1$,
the definitions of supremum and infimum yield that there exists a subvariety $Y$ of~$X_{\ol{K}}$, of dimension $d$ and not contained in $H\cup H_0$,
such that
\[
\sup_{\substack{s\in\mathrm{H}^0(X_{\ol{K}},L_{\ol{K}}^{\otimes n})\\n\in\NN\setminus\{0\}\\Y\nsubseteq \abs{\div(s)}}}\frac{I_{\ol{L}}(Y,s)}{\deg_{L}(Y)}\leq\widetilde{\eta}+\varepsilon.
\]
As $\deg_L(Y)>0$, such a subvariety $Y$ is contained in the set $X^{(d)}(\widetilde{\eta}+\varepsilon, \ol{L})$, but not in $H$, hence $X^{(d)}(\widetilde{\eta}+\varepsilon, \ol{L})\nsubseteq H$. As this is true for all closed subsets $H$ of $X_{\ol{K}}$ of codimension~$1$, we have that
\[
X^{(d)}(\widetilde{\eta}+\varepsilon,\ol{L})=X,
\]
which in turn implies, by the arbitrariness of~$\varepsilon$, that $\e_1^{(d)}(X,\ol{L})\leq\widetilde{\eta}$.

For the reverse inequality, consider again $\varepsilon>0$. By definition of the supremum, there exists a closed subset $H_\varepsilon$ of $X_{\ol{K}}$ of codimension~$1$, such that
\[
\inf_{\substack{Y\subseteq X_{\ol{K}}\\\dim(Y)=d\\ Y\nsubseteq H_\varepsilon\cup H_0}}\; \sup_{\substack{s\in\mathrm{H}^0(X_{\ol{K}},L_{\ol{K}}^{\otimes n})\\n\in\NN\setminus\{0\}\\Y\nsubseteq \abs{\div(s)}}}
\frac{I_{\ol{L}}(Y,s)}{\deg_{L}(Y)}>\widetilde{\eta}-\varepsilon;
\]
this means that for every subvariety $Y$ of dimension $d$ such that $Y\nsubseteq H_\varepsilon\cup H_0$,
\[
\sup_{\substack{s\in\mathrm{H}^0(X_{\ol{K}},L_{\ol{K}}^{\otimes n})\\n\in\NN\setminus\{0\}\\Y\nsubseteq \abs{\div(s)}}}\frac{I_{\ol{L}}(Y,s)}{\deg_{L}(Y)}>\widetilde{\eta}-\varepsilon.
\]
This implies that such subvarieties $Y$ can not satisfy the condition \eqref{eq:essentialminimum-ineq} for $\eta=\widetilde{\eta}-\varepsilon$; otherwise said, if a subvariety $Y$ satisfies the inequality in \eqref{eq:essentialminimum-ineq} with $\eta=\widetilde{\eta}-\varepsilon$, then it must be contained in $H_\varepsilon\cup H_0$ hence, by taking Zariski closures,
\[
X^{(d)}(\widetilde{\eta}-\varepsilon,\ol{L})\subseteq H_\varepsilon\cup H_0.
\]
As $H_\varepsilon\cup H_0$ has codimension~$1$, using \hyperref[first properties of successive minima]{Proposition~\ref*{first properties of successive minima}(3)} and the definition of the $d$-dimensional essential minimum forces $\widetilde{\eta}-\varepsilon\leq\e_1^{(d)}(X,\ol{L})$.
The arbitrariness of $\varepsilon>0$ yields $\widetilde{\eta}\leq\e_1^{(d)}(X,\ol{L})$, concluding the proof.
\end{proof}

%
%
%
%
%
%
%
%
%
%
%
%
%
%
%
%
%
%
%
%SECTION KEY INEQUALITY
%
%
%
%
%
%
%
%
%
%
%
%
%
%
%
%
%
%
%

\section{Key inequality and Zhang's inequality}\label{section key inequality}

Throughout this section, $X$ denotes a projective variety of dimension $N$ defined over a field~$K$.
Moreover, $L$ denotes a big semiample %(In the case of number fields, it is enough to consider nef.)
line bundle on~$X$.

The aim of this section is to prove the following key inequality, which is the essential ingredient for the main equidistribution result in this paper.
In addition, we also use it to prove an analogue of Zhang's inequality in our setting in Subsection 3.D.

\begin{theorem}[Key inequality]\label{thm:key-lemma}
Let $(\norm{\cdot}_v)_v$ be a semipositive quasi-algebraic metric on~$L$, and $\ol{\mathscr{O}_X}=(\mathscr{O}_X,(\norm{\cdot}_{v}')_v)$ be DSP quasi-algebraically metrized.
For $t\in\RR$, with $t$ close to $0$, and for $n>0$ big enough, there is a nonzero global section $s\in \mathrm{H}^0(X_{\ol{K}},L_{\ol{K}}^{\otimes n})$ satisfying
\begin{equation}\label{eq:very-small}
\sum_{v\in\fM}n_v\sup\log\norm{s}^{\otimes n}_{v}\norm{1}_{v}^{\prime\otimes t}\leq n\Big(- \widehat{\h}_{\ol{L}\otimes\ol{\mathscr{O}_X}^{\otimes t}}(X)+ O(t^2) \Big),
\end{equation}
where the implicit constant on $O(t^2)$ does not depend on~$n$.
\end{theorem}

To prove this theorem we deal with two different cases depending on the nature of~$K$.
In the function field case we mainly follow~\cite[\S5]{Gubler:2008}, whereas in the number field case we take Yuan's approach that was introduced in~\cite[\S3]{Yuan:2008}.
For the convenience of the reader, we shall nevertheless give the full argument of these proofs, to clear out the subtle differences.

\subsection{Function fields}
We can assume that $k$ is an algebraically closed field.
Let $K=k(C)$ be the function field of a regular projective curve $C$ over~$k$.

Let $\pi:\mathscr{X}\rightarrow C$ be a projective model of $X$,
and let $\mathscr{L}$ be a vertically nef model of $L^{\otimes e}$ on~$\mathscr{X}$, for some nonzero $e\in\NN$.
Let $\scrM$ be an ample line bundle on $C$ such that the metrized line bundle on $\ol{L}$ induced by the model $\scrL\otimes\pi^* \scrM$ is nef;
such a line bundle exists by \cite[Lemma~5.3]{Gubler:2008}.

\begin{proposition}\label{prop:functionfield}
Let $\scrN$ be a line bundle on $\scrX$ that is trivial on the generic fiber,
and assume $e=1$.
For every $t\in\QQ$, with $t$ close to~$0$, and every $r\in\QQ$ such that
\[
r> \frac{-\h_{\scrL\otimes \scrN^{\otimes t}}(X)}{(N+1)\deg_{\scrM}(C)\deg_L(X)}+O(t^2),
\]
with the implicit constant in $O(t^2)$ not depending on~$n$, we have
\[
h^0\big(\scrX,(\scrL\otimes \pi^*\scrM^{\otimes r}\otimes \scrN^{\otimes t})^{\otimes n}\big)>0,
\]
for every $n>0$ big enough, with $tn,rn\in \ZZ$.
\end{proposition}

\begin{proof}
First, as Gubler's consequence to Siu's theorem (see \cite[Lemma~5.6]{Gubler:2008}) we have the following inequality
\begin{multline}\label{eq:gubler}
h^0\big(\scrX,\big(\scrL\otimes(\pi^*\scrM)^{\otimes r}\otimes\scrN^{\otimes t}\big)^{\otimes n}\big)\\
\geq
\frac{1}{(N+1)!}\big(\deg_{\scrL\otimes(\pi^*\scrM)^{\otimes r}\otimes\scrN^{\otimes t}}(\scrX)+O(t^2)\big)n^{N+1} + o(n^{N+1}),
\end{multline}
for every $n\in\NN$ big enough, such that $tn,rn\in\ZZ$, and every $t\in\QQ$ close to~$0$.

To compute the degree appearing on the right-hand side of this inequality, one remarks the following.
Since $C$ is a curve, the intersection product of $\pi^*\scrM$ with itself is zero by the projection formula.
In addition, as $\scrN$ is trivial on the generic fiber of $\scrX$, we have that it is trivial on all but a finite number of fibers of $\pi$; therefore $\pi^*\scrM\cdot \scrN=0$ by Chow's moving lemma.
We use these observations, in conjunction with the  the multilinearity of the degree, to obtain
\[
\deg_{\scrL\otimes(\pi^*\scrM)^{\otimes r}\otimes \scrN^{\otimes t}}(\mathscr{X})=
\deg_{\scrL\otimes \scrN^{\otimes t}}(\mathscr{X}) + (N+1)\deg_{\scrL,\ldots,\scrL,(\pi^*\scrM)^{\otimes r}}(\mathscr{X}).
\]
Moreover, applying the projection formula to the second summand on the right-hand side, we have that
\[
\deg_{\scrL\otimes(\pi^*\scrM)^{\otimes r}\otimes \scrN^{\otimes t}}(\mathscr{X})=\deg_{\scrL\otimes \scrN^{\otimes t}}(\mathscr{X}) + (N+1)r\deg_{\scrM}(C)\deg_L(X).
\]

Finally, by Remark~\ref{rmk:height=degree} and inequality~\eqref{eq:gubler}, we get that for
\[
r>\frac{-\h_{\scrL\otimes\scrN^{\otimes t}}(X)}{(N+1)\deg_{\scrM}(C)\deg_L(X)}+O(t^2)
\]
and $n$ big enough, the statement holds.
\end{proof}

\begin{corollary}\label{key-ff}
Let $\scrN$ be a line bundle on $\scrX$ that is trivial on the generic fiber.
For every $t\in\mathbb{Q}$ close to $0$ and $n>0$ big enough, with $nt\in\mathbb{Z}$, there exists a nonzero global section $s\in\mathrm{H}^0(X,L^{\otimes n})$ such that
\[
\sum_{v\in \fM}n_v\log\sup\norm{s}^{\otimes n}_{\mathscr{L}\otimes\mathscr{N}^{\otimes t},v}\leq n\Big(-\widehat{\h}_{\mathscr{L}\otimes\mathscr{N}^{\otimes t}}(X)+O(t^2)\Big),
\]
with the implicit constant of $O(t^2)$ not depending on~$n$.
\end{corollary}
\begin{proof}
Without loss of generality, we can assume that $e=1$, since we can replace $L$ by $L^{\otimes e}$ which multiplies both sides of the equality by~$e$.

Let $\sum_{v}m_v\,v$ be a divisor on $C$ such that its associated line bundle is $\scrM$.
Notice that $\deg_{\scrM}(C)=\sum_v n_v m_v$.
For every nonzero global section $s\in \mathrm{H}^0(X,L^{\otimes n})$, we have
\[
\log\norm{s}^{\otimes n}_{\scrL\otimes(\pi^*\scrM)^{\otimes r}\otimes\scrN^{\otimes t},v}
=
\log\norm{s}^{\otimes n}_{\scrL\otimes\scrN^{\otimes t},v}-r\,n\,m_v.
\]
Hence
\[
\sum_{v\in \fM}n_v\log\sup\norm{s}^{\otimes n}_{\scrL\otimes(\pi^*\scrM)^{\otimes r}\otimes\scrN^{\otimes t},v}
=
\sum_{v\in \fM}n_v\log\sup\norm{s}^{\otimes n}_{\scrL\otimes\scrN^{\otimes t},v}-r\,n\,\deg_{\scrM}(C).
\]
Furthermore, the metric of a global section of the model $(\scrL\otimes(\pi^*\scrM)^{\otimes r}\otimes\scrN^{\otimes t})^{\otimes n}$ upper bounds this equality by~$0$, see Remark~\ref{remark:small sections over function fields}.
The existence of such a nonzero model section is guaranteed by Proposition~\ref{prop:functionfield} whenever
\[
r> \frac{-\h_{\scrL\otimes \scrN^{\otimes t}}(X)}{(N+1)\deg_{\scrM}(C)\deg_L(X)}+O(t^2).
\]
Hence, taking the limit
\[
r\longrightarrow  \bigg( \frac{-\h_{\scrL\otimes\scrN^{\otimes t}}(X)}{(N+1)\deg_{\scrM}(C)\deg_L(X)}+ O(t^2) \bigg) ^+
\]
concludes the proof.
\end{proof}

\subsection{Number fields}
In this subsection, let $L$ be only big and nef (not necessarily semiample). Fix an archimededan place $v_0\in\mathfrak{M}$.
Given a metric $(\norm{\cdot}_v)$ on $L$, we denote $\ol{L}(c):=\ol{L}(v_0,c/n_{v_0},1)$ for $c\in\RR$, as in Definition~\ref{perturbed line bundle}. By construction, if $\ol{L}$ is algebraic, then so is~$\ol{L}(c)$.

\begin{lemma}\label{key-nf}
Let $\ol{L}$ be algebraically metrized, $\ol{\mathscr{O}_X}$ be the trivial bundle on $X$ equipped with a DSP algebraic metric, and let $t\in\QQ$ close to~$0$.
Then, for $n$ big enough, there exists a nonzero section $s\in \mathrm{H}^0(X,L^{\otimes n})$ such that
\[
\log\sup\norm{s}^{\otimes n}_{\ol{L}\otimes \ol{\mathscr{O}_X}^{\otimes t},v}\leq
\left\lbrace
\begin{array}{ll}
n\Big(-\frac{1}{n_{v_0}}\widehat{\h}_{\ol{L}\otimes\ol{\mathscr{O}_X}^{\otimes t}}(X)+ O(t^2)\Big)&\mbox{ for }v=v_0,
\vspace{10pt}\\

0& \mbox{ for }v\neq v_0.
\end{array}\right.
\]

\end{lemma}

\begin{proof}
First, let us prove that $\ol{L}(c)\otimes\ol{\mathscr{O}_X}^{\otimes t}$ is big for
\begin{equation}\label{eq:NumberField-big-c}
c=-\widehat{\h}_{\ol{L}\otimes\ol{\mathscr{O}_X}^{\otimes t}}(X)
+ O(t^2)+\varepsilon,
\end{equation}
for every $\varepsilon>0$.

Let $(\scrX,\scrL,\scrN)$ be an algebraic model over $\cO_K$ of $(X,L^{\otimes e},\mathscr{O}_X)$, such that the non-archimedean $v$-adic metrics of $\ol{L}$ and $\ol{\mathscr{O}_X}$ are given by this model. This is possible after taking a common model~$\mathscr{X}$, see for instance \cite[Proposition~1.3.6]{Burgos-Philippon-Sombra:2014}.
Up to taking $\ol{\mathscr{O}_X}^{\otimes -1}$ we may assume that $t\geq 0$.
Let $\ol{N}_1,\ol{N}_2$ be two algebraic semipositive nef metrized line bundles on $X$, induced by line bundles $\scrN_1,\scrN_2$ on $\scrX$, such that $\ol{\mathscr{O}_X}\simeq\ol{N}_1\otimes\ol{N}_2^{\otimes -1}$; this comes from $\scrN\simeq\scrN_1\otimes\scrN_2^{\otimes -1}$.
Then, $\ol{L}_1=\ol{L}(c)\otimes \ol{N_1}^{\otimes t}$ and $\ol{L}_2=\ol{N}_2^{\otimes t}$ are two algebraic semipositive nef metrized $\QQ$-divisors such that
\[
\ol{L}(c)\otimes\ol{\scrO_X}^{\otimes t}\simeq \ol{L}_1\otimes\ol{L}_2^{\otimes -1}.
\]
We can apply Lemma~\ref{lemma:Siu-aritmetico} to obtain
\[
\widehat{\vol}(\ol{L}_1\otimes\ol{L}_2^{\otimes -1})
\geq \h_{\ol{L}_1}(X)-(N+1)\h_{\ol{L}_1,\ldots,\ol{L}_1,\ol{L}_2}(X).
\]
Moreover, by the multilinearity of the height function, we readily see that the right-hand side of this inequality amounts to
\[
\h_{\ol{L}(c)\otimes \ol{\scrO_X}^{\otimes t}}(X) + O(t^2).
\]
Then, for $c$ satisfying equality~\eqref{eq:NumberField-big-c} and by the multilinearity of the height,
we have $\h_{\ol{L}(c)\otimes\ol{\mathscr{O}_X}^{\otimes t}}(X)>0$.
This implies the arithmetic bigness of $\ol{L}(c)\otimes\ol{\mathscr{O}_X}^{\otimes t}$ by Lemma~\ref{lemma:Moriwaki}.

Next, since $\ol{L}(c)\otimes\ol{\mathscr{O}_X}^{\otimes t}$ is big, for $n$ big enough, there exists a nonzero section $s\in \mathrm{H}^0(X,L^{\otimes n})$ which is small with respect to the metrized line bundle $\big(\ol{L}(c)\otimes\ol{\scrO_X}^{\otimes t}\big)^{\otimes n}$.
The statement then follows from the equality
\[\log\sup\norm{s}^{\otimes n}_{\ol{L}(c)\otimes\ol{\scrO_X}^{\otimes t},v}=
\left\lbrace
\begin{array}{l}
\log\sup\norm{s}^{\otimes n}_{\ol{L}\otimes\ol{\scrO_X}^{\otimes t},v_0} + n\Bigg(\frac{1}{n_{v_0}}\widehat{\h}_{\ol{L}\otimes\ol{\scrO_X}^{\otimes t}}(X)+O(t^2)-\varepsilon\Bigg),
\vspace{10pt}\\

\log\sup\norm{s}^{\otimes n}_{\ol{L}\otimes\ol{\scrO_X}^{\otimes t},v}\qquad\quad \mbox{if }v\neq v_0,
\end{array}\right.
\]
and taking the limit $\varepsilon\rightarrow 0$.
\end{proof}

\subsection{Proof of Theorem~\ref{thm:key-lemma}}
Let us now prove the main theorem of this section, by using the previous results.
There is a double generalization step going on: first to consider quasi-algebraic metrics, and then to consider their limits.

\begin{proof}[Proof of Theorem~\ref{thm:key-lemma}]
When the metrics $\ol{L}$ and $\ol{\mathscr{O}_X}$ are algebraic and $t\in\QQ$, the result follows directly from Corollary~\ref{key-ff} and Lemma~\ref{key-nf}.

Next, assume $\ol{L}$ and $\ol{\mathscr{O}_X}$ are quasi-algebraic, still defined locally by models of some positive integer power of~$L$.
By Proposition~\ref{prop:panico}, up to a base change to a finite field extension of $K$, the metrics can still be taken as algebraic.
Since $v$-adic metrics (and henceforth heights) are invariant under finite field extensions, the result follows by the fact that $\sum_{w\vert v}n_w=n_v$.

It is left to prove the general case, which is obtained by limit from the above one. Let $t\in\RR$.
By definition, $\ol{L}=(L,(\norm{\cdot}_v))$
can be expressed as the uniform limit of quasi-algebraic semipositive metrized line bundles $\ol{L}_i$ whose $v$-adic metrics are defined by $K^{\circ}_v$-models of positive integer powers of~$L$.
The same holds for $\ol{\mathscr{O}_X}^{\otimes t}$, which can be approximated by a sequence $\ol{\mathscr{O}_{X,i}}^{\otimes t_i}$ of DSP quasi-algebraic metrized line bundles, with $t_i\in\QQ$.
Denote by $d_i=\mathrm{d}(\ol{L}\otimes \ol{\scrO_X}^{\otimes t},\ol{L}_i\otimes \ol{\scrO_{X,i}}^{\otimes t_i})$. 
By Lemma~\ref{lemma:Lipschitz}, there exists a Lipschitz continuous function in $d_i$ bounding
\[
\Abs{\widehat{\h}_{\ol{L}_i\otimes\ol{\mathscr{O}_{X,i}}^{\otimes t_i}}(X)-\widehat{\h}_{\ol{L}\otimes\ol{\mathscr{O}_X}^{\otimes t}}(X)}.
\]
By allowing $d_i$ to get as small as necessary compared to $t^2$, we obtain equation~\eqref{eq:very-small} in general.
\end{proof}

\subsection{Zhang's inequality}

As a first simple application of Theorem~\ref{thm:key-lemma}, we give the following analogue of Zhang's lower inequality on the essential minimum~\cite[Theorem~5.2]{Zhang:1995}.
\begin{corollary}\label{cor:Zhangs-ineq}
Let $X$ be a projective variety over~$K$, and $\ol{L}$ a big and semiample line bundle on $X$ equipped with a semipositive quasi-algebraic metric.
Then, for every integer $d=0,\ldots,\dim(X)$, we have
\[
\e_1^{(d)}(X,\ol{L})\geq\widehat{\h}_{\ol{L}}(X).
\]
\end{corollary}

\begin{proof}
Equip $\mathscr{O}_X$ with the trivial metric. For every $\varepsilon>0$, choose $t$ close enough to $0$ such that, by Theorem~\ref{thm:key-lemma}, there exists $s\in \mathrm{H}^0(X_{\ol{K}},L_{\ol{K}}^{\otimes n})$ for $n$ big enough satisfying
\[
\frac{I_{\overline{L}}(Y,s)}{\deg_L(Y)}\geq\widehat{\h}_{\ol{L}}(X)-\varepsilon
\]
for every $d$-dimensional subvariety $Y$ of $X_{\overline{K}}$ that does not lie neither in $\abs{\div(s)}$ nor in the Zariski closed subset $H_0$ defined in \hyperref[prop:generic-big]{Proposition \ref*{prop:generic-big}}.
Therefore, by \hyperref[prop:equivalent-def-ess-minimum]{Proposition \ref*{prop:equivalent-def-ess-minimum}} we obtain the result.
\end{proof}

Notice that in the case of function fields (and $d=0$), Gubler~\cite{Gubler:2008} gave a proof of Zhang's inequality that also allowed to consider the case of function fields of higher dimensional varieties.
However, his argument relies on a reduction to the function field case of curves that strongly uses the fact that $\e_1^{(0)}(X,\ol{L})$ (the essential minimum of points) does not depend on the choice of the section, bypassing the need of a ``Key inequality'' (of the form of Theorem~\ref{thm:key-lemma}) that holds for function fields of higher dimensional varieties.
It is our impression that, in the setting of this article, we cannot avoid Theorem~\ref{thm:key-lemma} to prove Corollary~\ref{cor:Zhangs-ineq} in this case.

%
%
%
%
%
%
%
%
%
%
%
%
%
%
%
%
%
%
%
%SECTION EQUIDISTRIBUTION
%
%
%
%
%
%
%
%
%
%
%
%
%
%
%
%
%
%
%

\section{Equidistribution of small effective cycles}\label{section 4}

Fix for the entire section the choice of a projective variety $X$ over~$K$.
A cycle of $X_{\ol{K}}$ is an element of the free abelian group generated by the subvarieties of~$X_{\ol{K}}$.
Recall that it is said to be \emph{effective} if all of its coefficients are nonnegative, it is called \emph{of pure dimension $d$} (or a \emph{$d$-cycle}) if it is a linear combination of subvarieties of $X_{\ol{K}}$ of dimension~$d$, and it is said to be \emph{Galois invariant} if it is fixed by all the elements of~$\Gal(\ol{K}/K)$.

In this section, we are interested in the study of the interactions between the geometric and arithmetic properties of a net $(Z_m)_m$ of Galois invariant effective $d$-cycles of~$X$.
In this case, we write each member of the net as the formal finite sum
\begin{equation}\label{cycles decomposition}
Z_m=\sum_ia_{m,i}\, Y_{m,i},
\end{equation}
with $a_{m,i}\in\mathbb{N}$ and $Y_{m,i}$ being a subvariety of $X_{\ol{K}}$ of dimension $d$ for all~$i$.
After defining the notions of genericity and smallness for such nets, we focus on the statement and proof of the equidistribution theorem by means of the higher dimensional essential minima introduced above.

\subsection{Generic and small nets of cycles}

The two following definitions for nets of cycles over a directed set $(\mathcal{J},\preceq)$ are fundamental for equidistribution statements.
The first one is of geometric nature and formalizes the requirement that the members of the net have a negligible summand in any closed subset of $X_{\ol{K}}$ of codimension~$1$.

\begin{definition}\label{definition genericity}
Let $L$ be a big and nef line bundle on~$X$. 
A net $(Z_m)_{m\in\mathcal{J}}$ of cycles of $X_{\ol{K}}$ is said to be \emph{$L$-generic} if the degree of $Z_m$ with respect to $L$ is eventually nonzero and, for every closed subset $H\subseteq X_{\ol{K}}$ of codimension~$1$, we have that
\[
\lim_{m}\frac{1}{\deg_L(Z_m)}\sum_{Y_{m,i}\subseteq H}a_{m,i}\deg_L(Y_{m,i})=0,
\]
with notation as in \eqref{cycles decomposition}.
\end{definition}

Notice that, if $(Y_m)_{m}$ is a net of subvarieties of $X_{\ol{K}}$, \hyperref[prop:generic-big]{Proposition~\ref*{prop:generic-big}} assures that the previous definition is equivalent to the following statement: for every closed subset $H$ of $X_{\ol{K}}$ of codimension~$1$ there exists $m_0\in\mathcal{J}$ such that $Y_m\nsubseteq H$ for all~$m\succeq m_0$.
In such a case, and as in \cite[6.2]{Gubler:2008}, the net $(Y_m)_{m}$ is simply called \emph{generic}, to underline its independence on the choice of $L$.
This agrees with \hyperref[def introduction generic and small]{Definition~\ref*{def introduction generic and small}} in the Introduction.

\begin{remark}
A sequence $(Y_m)_{m}$ of subvarieties of $X_{\ol{K}}$ is generic if and only if for every closed subset $H\subseteq X_{\ol{K}}$ of codimension~$1$, the set $\{m\in\mathbb{N}\,\mid\,Y_m\subseteq H\}$ is finite, which agrees with the classical definition.
\end{remark}

\begin{remark}\label{rmk:genericity for Galois cycle}
Let $Y$ be a subvariety of $X_{\ol{K}}$ of dimension~$d$.
For every element $\sigma$ of the absolute Galois group of~$K$, denote by $Y^\sigma$ the corresponding Galois conjugate of~$Y$; it is again a subvariety of $X_{\ol{K}}$ of dimension~$d$.
The finite set $O(Y):=\{Y^\sigma\mid\sigma\in\Gal(\ol{K}/K)\}$ is called the \emph{Galois orbit} of $Y$ over $K$ and the cycle
\[
Y^{\Gal}:=\sum_{Y^\sigma\in O(Y)}Y^\sigma
\]
is called the \emph{Galois cycle} of~$Y$.
It is a Galois invariant $d$-cycle of $X_{\ol{K}}$ by construction, with degree~$\#O(Y)\,\deg_L(Y)$.

A net $(Y_m)_m$ of subvarieties of $X_{\ol{K}}$ is $L$-generic if and only if the net of their associated Galois cycles is.
Indeed, if $Y_m$ lies in a one codimensional closed subset $H$ of~$X_{\ol{K}}$, then the support of $Y_m^{\Gal}$ is contained in the union of Galois conjugates of~$H$.
Conversely, if $Y_m^{\Gal}$ has a summand contained in~$H$, then $Y_m$ lies in the union of Galois conjugates of~$H$.
\end{remark}

The notion of genericity of a net of $d$-dimensional subvarieties of $X_{\ol{K}}$ is intimately related with the $d$-dimensional essential minimum of a semipositive quasi-algebraic metrized line bundle, as the next statement shows.

\begin{proposition}\label{prop:inequality liminf-essential minimum}
Let $\ol{L}$ be a semipositive quasi-algebraic metrized line bundle on $X$ with $L$ big and nef.
Let $d=0,\ldots,\dim(X)$, and $(Y_m)_{m\in\mathcal{J}}$ be a generic net of $d$-dimensional subvarieties of~$X_{\ol{K}}$.
Then,
\begin{equation}\label{inequality liminf-essential minimum}
\liminf_{m}\sup_{\substack{s\in\mathrm{H}^0(X_{\ol{K}},L_{\ol{K}}^{\otimes n})\\n\in\mathbb{N}\setminus\{0\}\\Y_m\nsubseteq|\div(s)|}}\frac{I_{\ol{L}}(Y_m,s)}{\deg_L(Y_m)}\geq\e_1^{(d)}(X,\ol{L}).
\end{equation}
Moreover, there exists a generic net $(Y_m)_{m\in\mathcal{J}}$ of $d$-dimensional subvarieties of $X_{\ol{K}}$ satisfying
\begin{equation}\label{equality liminf-essential minimum}
\lim_{m}\sup_{\substack{s\in\mathrm{H}^0(X_{\ol{K}},L_{\ol{K}}^{\otimes n})\\n\in\mathbb{N}\setminus\{0\}\\Y_m\nsubseteq|\div(s)|}}\frac{I_{\ol{L}}(Y_m,s)}{\deg_L(Y_m)}=\e_1^{(d)}(X,\ol{L}).
\end{equation}
\end{proposition}
\begin{proof}
Let $H_0$ denote the closed subset of $X_{\ol{K}}$ introduced in \hyperref[prop:generic-big]{Proposition~\ref*{prop:generic-big}}.
For simplicity of notation, write
\[
F(Y):=\sup_{\substack{s\in\mathrm{H}^0(X_{\ol{K}},L_{\ol{K}}^{\otimes n})\\n\in\mathbb{N}\setminus\{0\}\\Y\nsubseteq|\div(s)|}}\frac{I_{\ol{L}}(Y,s)}{\deg_L(Y)}
\]
for every subvariety $Y$ of $X_{\ol{K}}$ such that~$Y\nsubseteq H_0$.
Let $H$ be any closed subset of $X_{\ol{K}}$ of codimension~$1$.
By genericity of the net $(Y_m)_{m}$, there exists $k_H\in\mathcal{J}$ such that the subvariety $Y_m$ is not contained in $H\cup H_0$ whenever~$m\succeq k_H$.
Hence,
\[
\liminf_mF(Y_m)=\adjustlimits\sup_k\inf_{m\succeq k}F(Y_m)\geq\inf_{m\succeq k_H}F(Y_m)\geq\inf_{\substack{Y\subseteq X_{\ol{K}}\\\dim(Y)=d\\Y\nsubseteq H\cup H_0}}F(Y).
\]
As this is true for all choice of~$H$, \hyperref[prop:equivalent-def-ess-minimum]{Proposition~\ref*{prop:equivalent-def-ess-minimum}} implies \eqref{inequality liminf-essential minimum}.

For the second claim, consider the directed set $(\mathcal{J},\subseteq)$ consisting of all closed subsets of $X_{\ol{K}}$ of pure codimension $1$ and containing~$H_0$, endowed with the usual inclusion relation.
For every~$H\in\mathcal{J}$, \hyperref[prop:equivalent-def-ess-minimum]{Proposition~\ref*{prop:equivalent-def-ess-minimum}} ensures that
\[
\inf_{\substack{Y\subseteq X_{\ol{K}}\\\dim(Y)=d\\Y\nsubseteq H}}F(Y)\leq\e_1^{(d)}(X,\ol{L}).
\]
Hence, there exists a subvariety $Y_H\subseteq X_{\ol{K}}$ of dimension $d$, not contained in $H$ and satisfying
\begin{equation}\label{recursive definition of small sequence}
F(Y_H)\leq\e_1^{(d)}(X,\ol{L})+\frac{1}{\ell_H},
\end{equation}
with $\ell_H$ being the number of irreducible components of~$H$. 
Notice that the function $\ell_{\bullet}$ from $(\mathcal{J},\subseteq)$ to $\mathbb{N}$ is strictly increasing.
Therefore, combining \eqref{inequality liminf-essential minimum} and~\eqref{recursive definition of small sequence},
\[
\e_1^{(d)}(X,\ol{L})\leq\liminf_H F(Y_H)\leq\limsup_H F(Y_H)\leq\e_1^{(d)}(X,\ol{L}).
\]
This shows that the net $(Y_H)_{H\in\mathcal{J}}$ satisfies \eqref{equality liminf-essential minimum}.
Moreover, for every closed subset $H$ of $X_{\ol{K}}$ of codimension~$1$ consider $H^\prime\in\mathcal{J}$ such that $H\subseteq H^\prime$.
By construction, $Y_{H^\prime}\nsubseteq H$, so the definition of the preorder on~$\mathcal{J}$ implies that the net $(Y_H)_{H\in\mathcal{J}}$ is generic. 
\end{proof}

\begin{remark}
When the base field $K$ is countable (for instance when $K$ is a number field), equality \eqref{equality liminf-essential minimum} holds for a generic sequence $(Y_m)_{m\in\NN}$ of $d$-dimensional subvarieties of~$X_{\ol{K}}$.
Indeed, in such a case the collection of irreducible closed subsets of $X_{\ol{K}}$ of pure codimension $1$ is countable.
One can write it as $\{H_1,H_2,\dots\}$ and assume that~$H_0\subseteq H_1$, where $H_0$ is the closed subset of \hyperref[prop:generic-big]{Proposition~\ref*{prop:generic-big}}.
To obtain the claim, it suffices to repeat the argument in the previous proof by taking $\mathcal{J}$ to be the countable family whose $k$-th element is~$H_1\cup\ldots\cup H_k$.
\end{remark}

\begin{remark}\label{rmk:another equivalent definition}
The relations proven in \hyperref[prop:inequality liminf-essential minimum]{Proposition~\ref*{prop:inequality liminf-essential minimum}} give a third equivalent definition for the $d$-dimensional essential minimum of a semipositive quasi-algebraic metrized line bundle $\ol{L}$ with $L$ big and nef, that is
\[
\e_1^{(d)}(X,\ol{L})=\min_{(Y_m)_m}\;\liminf_m\sup_{\substack{s\in\mathrm{H}^0(X_{\ol{K}},L_{\ol{K}}^{\otimes n})\\n\in\mathbb{N}\setminus\{0\}\\Y_m\nsubseteq|\div(s)|}}\frac{I_{\ol{L}}(Y_m,s)}{\deg_L(Y_m)},
\]
where the minimum is taken over the family of generic nets $(Y_m)_{m\in\mathcal{J}}$ of $d$-dimensional subvarieties of~$X_{\ol{K}}$.
\end{remark}

The result of \hyperref[prop:inequality liminf-essential minimum]{Proposition~\ref*{prop:inequality liminf-essential minimum}} suggests the following arithmetic notion for a net of effective cycle.

\begin{definition}\label{definition smallness}
Let $\ol{L}$ be a semipositive quasi-algebraic metrized line bundle on $X$ and~$d=0,\dots,\dim(X)$.
A net $(Z_m)_{m}$ of effective cycles of pure dimension $d$ in~$X_{\ol{K}}$, whose degrees with respect to $L$ are eventually nonzero, is said to be \emph{$\overline{L}$-small} if
\[
\lim_m\frac{1}{\deg_L(Z_m)}\sum_ia_{m,i}\sup_{\substack{s\in\mathrm{H}^0(X_{\ol{K}},L_{\ol{K}}^{\otimes n})\\n\in\mathbb{N}\setminus\{0\}\\Y_{m,i}\nsubseteq|\div(s)|}}I_{\ol{L}}(Y_{m,i},s)=\e_1^{(d)}(X,\ol{L}),
\]
with notation as in~\eqref{cycles decomposition}.
\end{definition}

In particular, a generic net of $d$-dimensional subvarieties is $\overline{L}$-small if and only if it satisfies~\eqref{equality liminf-essential minimum},
which agrees with Definition~\ref{def introduction generic and small} in the Introduction.
Loosely speaking, this is equivalent to the requirement that the asymptotic behavior of the (normalized) maximal ``height-gap'' of its members is as small as possible, which justifies the adopted terminology.

\begin{remark}
A net $(p_m)_{m}$ of closed points in $X_{\ol{K}}$ is $\overline{L}$-small if and only if
\[
\lim_m\h_{\overline{L}}(p_m)=\e_1^{(0)}(X,\ol{L}).
\]
This agrees, for sequences, with the classical definition.
A comparison to previous notions of smallness for higher dimensional subvarieties is carried out in \hyperref[subsection comparison]{subsection~5.A}.
\end{remark}

\begin{remark}\label{rmk:smallness for Galois cycle}
A net $(Y_m)_m$ of $d$-dimensional subvarieties of $X_{\ol{K}}$ is $\ol{L}$-small if and only if the net of their associated Galois cycles, as in \hyperref[rmk:genericity for Galois cycle]{Remark~\ref*{rmk:genericity for Galois cycle}}, is such.
This follows from the observation that, for every Galois automorphism $\sigma$ of $\ol{K}$, we have
\[
\sup_{\substack{s\in\mathrm{H}^0(X_{\ol{K}},L_{\ol{K}}^{\otimes n})\\n\in\mathbb{N}\setminus\{0\}\\Y_m\nsubseteq|\div(s)|}}I_{\ol{L}}(Y_m,s)
=
\sup_{\substack{s\in\mathrm{H}^0(X_{\ol{K}},L_{\ol{K}}^{\otimes n})\\n\in\mathbb{N}\setminus\{0\}\\Y^\sigma_m\nsubseteq|\div(s)|}}I_{\ol{L}}(Y^\sigma_m,s),
\]
since $L$ is defined on $K$ and correcting integrals are invariant under Galois action (using for instance \hyperref[correcting integrals as height jumps]{Remark~\ref*{correcting integrals as height jumps}} and the analogous property for heights).
\end{remark}

\begin{remark}\label{rmk:essential minimum and small net of Galois invariant cycle}
Let $(Z_m)_m$ be a $\ol{L}$-small net of Galois invariant effective $d$-cycles in~$X_{\ol{K}}$.
Then, each member of the net can be written as $Z_m=\sum_ia_{m,i}Y_{m,i}^{\Gal}$, with each $Y_{m,i}$ being a $d$-dimensional subvariety of~$X_{\ol{K}}$, and $Y_{m,i}^{\Gal}$ the corresponding Galois cycle as in \hyperref[rmk:genericity for Galois cycle]{Remark~\ref*{rmk:genericity for Galois cycle}}.
The $\ol{L}$-smallness and the effectiveness of the net yield
\begin{equation*}
\begin{split}
\e_1^{(d)}(X,\ol{L})
&=
\limsup_m\frac{1}{\deg_L(Z_m)}\sum_ia_{m,i}\sum_{Y^\sigma_{m,i}\in O(Y_{m,i})}\sup_{\substack{s\in\mathrm{H}^0(X_{\ol{K}},L_{\ol{K}}^{\otimes n})\\Y^\sigma_{m,i}\nsubseteq|\div(s)|}} I_{\ol{L}}(Y_{m,i}^\sigma,s)
\\&\geq
\limsup_m\frac{1}{\deg_L(Z_m)}\sum_ia_{m,i}\,\sup_s I_{\ol{L}}\big(Y_{m,i}^{\Gal},s\big),
\end{split}
\end{equation*}
where the last supremum is taken over the global sections $s$ of tensor powers of $L_{\ol{K}}$ which intersect $Y^{\Gal}_{m,i}$ properly.
As the sections are not necessarily defined over~$K$, the previous inequality may be strict; this is a reason for \hyperref[definition smallness]{Definition~\ref*{definition smallness}}.
\end{remark}

It is clear that for a net of $d$-dimensional effective cycles of $X_{\ol{K}}$ being generic and small are unrelated requirements.
For instance, if $X$ is the projective line over the field of rational numbers and $\ol{L}$ is the line bundle $\mathscr{O}(1)$ equipped with the canonical metric, the sequence which is constantly equal to the point $[1:1]$ is $\ol{L}$-small but not generic, whereas the sequence $([1:m])_{m\in\NN}$ is generic but not $\ol{L}$-small.
The equidistribution result of next section concerns the nets of effective cycles which are both generic and small.

\subsection{Equidistribution theorems}

For a topological space~$X$, denote by $\mathcal{C}_b(X,\mathbb{R})$ the real vector space of bounded continuous real-valued functions on~$X$.
Recall that a net $(\mu_m)_{m}$ of Borel probability measures on $X$ is said to \emph{converge weakly} to another Borel probability measure $\mu$ on $X$ if
\begin{equation}\label{convergence of integrals}
\int_Xf\ d\mu_m\longrightarrow\int_Xf\ d\mu
\end{equation}
for every~$f\in\mathcal{C}_b(X,\mathbb{R})$.
The following criterion allows to prove such a weak convergence by only checking \eqref{convergence of integrals} for a big enough family of bounded continuous functions on~$X$.

\begin{proposition}[Weyl's criterion]\label{Weyl criterion}
Let $X$ be a topological space, and $\mathcal{A}$ be a subset of $\mathcal{C}_b(X,\mathbb{R})$ such that the vector subspace generated by $\mathcal{A}$ is dense in $\mathcal{C}_b(X,\mathbb{R})$ with respect to the uniform convergence topology.
Then, a net of Borel probability measures $(\mu_m)_{m\in\mathcal{J}}$ on $X$ converges weakly to a Borel probability measure $\mu$ if and only if the convergence in \eqref{convergence of integrals} holds for every~$f\in\mathcal{A}$.
\end{proposition}
\begin{proof}
One direction is obvious.
For the converse, fix~$f\in\mathcal{C}_b(X,\mathbb{R})$.
By density, for every $\varepsilon>0$ there exists a linear combination $f_{\varepsilon}$ of elements of $\mathcal{A}$ such that~$\|f-f_\varepsilon\|_{\sup}<\varepsilon/3$.
The linearity of the integral assures that \eqref{convergence of integrals} holds for~$f_\varepsilon$, then by hypothesis there exists $m_0\in\mathcal{J}$ such that
\[
\bigg|\int_Xf_\varepsilon\ d\mu_m-\int_Xf_\varepsilon\ d\mu\bigg|<\frac{\varepsilon}{3}
\]
for every~$m\succeq m_0$.
It follows that, for any such $m\in\mathcal{J}$,
\begin{multline*}
\bigg|\int_Xf\ d\mu_m-\int_Xf d\mu\bigg|
\leq
\bigg|\int_X(f-f_\varepsilon)\ d\mu_m\bigg|+
\\+\bigg|\int_Xf_\varepsilon\ d\mu_m-\int_Xf_\varepsilon\ d\mu\bigg|+\bigg|\int_X(f_\varepsilon-f)\ d\mu\bigg|
<\varepsilon,
\end{multline*}
which verifies the claim.
\end{proof}

Consider a Galois invariant effective $d$-cycle $Z$ of~$X_{\ol{K}}$.
By grouping together Galois orbits of subvarieties of~$\ol{K}$, $Z$ can be seen as a cycle of~$X$, see~\cite[A.4.13]{Bombieri-Gubler:2006}.
This allows to consider, if $v\in\mathfrak{M}$ and $\ol{L}_v$ is a semipositive $v$-adic metrized line bundle on~$X$, the measure
\[
\c_1(\ol{L})^{\wedge d}\wedge\delta_{Z_v^{\an}}
\]
on $X_v^{\an}$ defined in~\eqref{eq: Chambert-Loir measure}.
It is positive and of total mass $\deg_L(Z)$ by \hyperref[prop:measure-properties]{Proposition~\ref*{prop:measure-properties}}, and it is independent on the choice of the embedding $K_v\hookrightarrow\ol{K}_v$ thanks to the Galois invariancy.

We are ready to prove the main result of the paper, that is an equidistribution theorem for small and generic nets of Galois invariant effective cycles of~$X_{\ol{K}}$.
Its proof is inspired by the classical strategy and involves suitable perturbations of metrized line bundles, as well as the \hyperref[thm:key-lemma]{Key inequality} of \hyperref[section key inequality]{Section \ref*{section key inequality}}.

\begin{theorem}[equidistribution of effective cycles]\label{equidistribution theorem}
Let $X$ be a projective variety over~$K$, and $\ol{L}$ be a big and semiample line bundle on $X$ equipped with a quasi-algebraic semipositive metric.
Let also $d=0,\dots,\dim(X)$ and assume that $\e_1^{(d)}(X,\ol{L})=\widehat{\h}_{\overline{L}}(X)$.
Then, for every $L$-generic and $\overline{L}$-small net $(Z_m)_{m\in\mathcal{J}}$ of Galois invariant effective cycles of $X_{\ol{K}}$ of pure dimension $d$, the weak convergence of probability measures on~$X_v^{\an}$
\[
\frac{1}{\deg_L(Z_m)}\ \c_1(\overline{L}_v)^{\wedge d}\wedge\delta_{Z_{m,v}^{\an}}\longrightarrow \frac{1}{\deg_L(X)}\ \c_1(\overline{L}_v)^{\wedge\dim(X)}
\]
holds for every $v\in\mathfrak{M}$.
\end{theorem}
\begin{proof}

%
%preamble and reductions
%

Let $(Z_m)_m$ be a $L$-generic and $\overline{L}$-small net of Galois invariant effective $d$-cycles of $X_{\ol{K}}$ of pure dimension~$d$.
By Galois invariancy, we can write each member of the net as
\[
Z_m=\sum_ia_{m,i}\,Y_{m,i},
\]
with every $Y_{m,i}$ being a subvariety of~$X$.
By definition of $L$-genericity, up to considering a queue of the net, we can assume that $\deg_L(Z_m)>0$ for all $m\in\mathcal{J}$.
Moreover, for every $v\in\mathfrak{M}$, the probability measure
\begin{equation}\label{measure in the proof of equidistribution}
\frac{1}{\deg_L(Z_m)}\ c_1(\overline{L}_v)^{\wedge d}\wedge\delta_{Z_{m,v}^{\an}}
\end{equation}
is not affected by removing from the cycle $Z_m$ the subvarieties $Y_{m,i}$ whose degree with respect to $L$ vanishes, because of \hyperref[prop:measure-properties]{Proposition \ref*{prop:measure-properties} (1)}.
Hence, we can also assume that $\deg_L(Y_{m,i})>0$ for all $i$ and $m\in\mathcal{J}$.
Therefore, the measure \eqref{measure in the proof of equidistribution} can be written as
\[
\sum_i\frac{a_{m,i}\;\deg_L(Y_{m,i})}{\deg_L(Z_m)}\frac{1}{\deg_L(Y_{m,i})}\ \c_1(\ol{L}_v)^{\wedge d}\wedge\delta_{Y_{m,i,v}^{\an}}.
\]

Fix a place~$v_0\in\mathfrak{M}$.
Since $X$ is proper, the analytic space $X_{v_0}^{\an}$ is compact by~\cite[Theorem~3.4.8 (ii)]{Berkovich:1990}. 
So, because of \hyperref[density of brick functions]{Theorem \ref*{density of brick functions}} and \hyperref[Weyl criterion]{Proposition~\ref*{Weyl criterion}}, we can reduce to prove the convergence \eqref{convergence of integrals} for $v_0$-adic elementary functions.

%
%exploiting the smallness
%

From now on, let $f$ be a $v_0$-adic elementary function, and consider $t\in\mathbb{Q}$ sufficiently close to~$0$.
By \hyperref[perturbation of DSP qa is DSP qa after a finite field extension]{Lemma~\ref*{perturbation of DSP qa is DSP qa after a finite field extension}}, there exists a finite field extension $K^\prime$ of $K$ and a place $w_0$ of $K^\prime$ dividing $v_0$ such that $\overline{L}(w_0,f,t)$ is a DSP quasi-algebraic metrized line bundle defined over~$K'$.
The key inequality of \hyperref[thm:key-lemma]{Theorem~\ref*{thm:key-lemma}} asserts that there exists $n\in\mathbb{N}$ and a global section $s_t\in H^0(X_{\ol{K}},L^{\otimes n}_{\ol{K}})$ for which
\begin{equation}\label{defining property of the supersmall section in the proof of equidistribution}
\sum_{w\in\mathfrak{M}_{K^\prime}}n_w\sup\log\|s_t\|_{\ol{L}(w_0,f,t),w}\leq n\Big(-\widehat{\h}_{\ol{L}(w_0,f,t)}(X)+O(t^2)\Big).
\end{equation}

To simplify the notation in the remaining of the proof, set, for every subvariety $Y$ of~$X$,
\[
\mathcal{H}(Y):=\big\{s\in H^0(X_{\ol{K}},L_{\ol{K}}^{\otimes n})\mid\, n\in\mathbb{N}\setminus\{0\}\text{ and } s\text{ intersects }Y\text{ properly}\big\}.
\] 

By \hyperref[rmk:essential minimum and small net of Galois invariant cycle]{Remark~\ref*{rmk:essential minimum and small net of Galois invariant cycle}}, the $\ol{L}$-smallness and the effectiveness of the net $(Z_m)_m$ ensure that
\begin{multline}\label{starting inequality in the proof of equidistribution}
\e_1^{(d)}(X,\ol{L})\geq
\limsup_{m}\frac{1}{\deg_L(Z_m)}\sum_ia_{m,i}\sup_{s\in\mathcal{H}(Y_{m,i})}I_{\ol{L}}(Y_{m,i},s)\\
=
\limsup_{m}\Bigg(\sum_{\substack{i\\s_t\in\mathcal{H}(Y_{m,i})}}\frac{a_{m,i}\deg_L(Y_{m,i})}{\deg_L(Z_m)}\sup_{s\in\mathcal{H}(Y_{m,i})}\frac{I_{\ol{L}}(Y_{m,i},s)}{\deg_L(Y_{m,i})}
\\
+\sum_{\substack{i\\s_t\notin\mathcal{H}(Y_{m,i})}}\frac{a_{m,i}\deg_L(Y_{m,i})}{\deg_L(Z_m)}\sup_{s\in\mathcal{H}(Y_{m,i})}\frac{I_{\ol{L}}(Y_{m,i},s)}{\deg_L(Y_{m,i})}\Bigg).
\end{multline}

%
%lower bound for the part not contained in the supersmall divisor
%

Observe that, if $s_t\in\mathcal{H}(Y_{m,i})$ then it intersects $Y_{m,i}$ properly, so a combination of the definition of correcting integral, \hyperref[correcting integrals for perturbed metrics]{Lemma \ref*{correcting integrals for perturbed metrics}}, \eqref{defining property of the supersmall section in the proof of equidistribution} and \hyperref[prop:measure-properties]{Proposition~\ref*{prop:measure-properties}(1)} yields
\begin{multline}\label{supersmall section in the proof of equidistribution}
\sup_{s\in\mathcal{H}(Y_{m,i})}\frac{I_{\ol{L}}(Y_{m,i},s)}{\deg_L(Y_{m,i})}\geq\frac{I_{\ol{L}}(Y_{m,i},s_t)}{\deg_L(Y_{m,i})}
\\
\geq \big(\,\widehat{\h}_{\ol{L}(w_0,f,t)}(X)-O(t^2)\big)
-\frac{n_{w_0}t}{\deg_L(Y_{m,i})}\int_{X_{v_0}^{\an}}f\ \c_1(\ol{L}_{v_0})^{\wedge d}\wedge\delta_{Y_{m,i,v_0}^{\an}}.
\end{multline}

%
%lower bound for the part contained in the supersmall divisor
%

To control the second summand in~\eqref{starting inequality in the proof of equidistribution}, since $L$ is semiample, \hyperref[lemma about lower absolute constant]{Lemma \ref*{lemma about lower absolute constant}} ensures that there exists a real constant $C_{\ol{L}}$ for which, for all the $d$-dimensional subvarieties $Y_{m,i}$ of~$X$,
\begin{multline}\label{lower bound in the proof of equidistribution}
\sup_{s\in\mathcal{H}(Y_{m,i})}\frac{I_{\ol{L}}(Y_{m,i},s)}{\deg_L(Y_{m,i})}\geq C_{\ol{L}}
\\
\geq\big(C_{\ol{L}}+n_{w_0}t\cdot\min f\big)-\frac{n_{w_0}t}{\deg_L(Y_{m,i})}\int_{X_{v_0}^{\an}}f\ \c_1(\ol{L}_{v_0})^{\wedge d}\wedge\delta_{Y_{m,i,v_0}^{\an}},
\end{multline}
where the second inequality comes from the fact that the involved measure is positive and of total mass~$\deg_L(Y_{m,i})$.

%
%putting together the pieces
%

Plugging inequalities \eqref{supersmall section in the proof of equidistribution} and \eqref{lower bound in the proof of equidistribution} in \eqref{starting inequality in the proof of equidistribution} and using the fact that the coefficients $a_{m,i}$ are nonnegative, one has after reordering the terms
\begin{multline*}
\e_1^{(d)}(X,\ol{L})
\geq
\limsup_{m}\Bigg(\frac{-n_{w_0}t}{\deg_L(Z_m)}\sum_{\substack{i\\\text{}}}a_{m,i}\int_{X_{v_0}^{\an}}f\ \c_1(\ol{L}_{v_0})^{\wedge d}\wedge\delta_{Y_{m,i,v_0}^{\an}}
\\
+\big(\widehat{\h}_{\ol{L}(w_0,f,t)}(X)-O(t^2)\big)\sum_{\substack{i\\s_t\in\mathcal{H}(Y_{m,i})}}\frac{a_{m,i}\deg_L(Y_{m,i})}{\deg_L(Z_m)}
\\
+\big(C_{\ol{L}}+n_{w_0}t\cdot\min f\big)\sum_{\substack{i\\s_t\notin\mathcal{H}(Y_{m,i})}}\frac{a_{m,i}\deg_L(Y_{m,i})}{\deg_L(Z_m)}\Bigg).
\end{multline*}

%
%exploiting genericity
%

Consider the union of the Galois conjugates of~$|\div(s_t)|$.
It is a closed subset of $X$ of codimension~$1$; by construction, $s_t$ intersects $Y_{m,i}$ properly if $Y_{m,i}$ is not contained in it.
Then, the definition of $L$-genericity implies that the second and third summands in the right hand side of the above inequality admit a limit with respect to~$m$.
These limits are, respectively, $\widehat{\h}_{\ol{L}(w_0,f,t)}(X)-O(t^2)$ and~$0$.
Therefore we obtain
\begin{equation}\label{final inequality for essmin in the proof of equidistribution}
\e_1^{(d)}(X,\ol{L})\geq
\limsup_{m}\Bigg(\frac{-n_{w_0}t}{\deg_L(Z_m)}\int_{X_{v_0}^{\an}}f\ \c_1(\ol{L}_{v_0})^{\wedge d}\wedge\delta_{Z_{m,v_0}^{\an}}\Bigg)
+\widehat{\h}_{\ol{L}(w_0,f,t)}(X)-O(t^2).
\end{equation}

%
%concluding
%

The hypothesis $\e_1^{(d)}(X,\ol{L})=\widehat{\h}_{\overline{L}}(X)$ and an application of \hyperref[lemma about perturbed height]{Lemma \ref*{lemma about perturbed height}} yield
\begin{equation}\label{equality for the essmin in the proof of equidistribution}
\e_1^{(d)}(X,\ol{L})=\widehat{\h}_{\overline{L}(w_0,f,t)}(X)-\frac{n_{w_0}t}{\deg_L(X)}\int_{X^{\an}_{v_0}}f\ c_1(\overline{L}_{v_0})^{\wedge\dim(X)}+O(t^2).
\end{equation}

It suffices to combine \eqref{final inequality for essmin in the proof of equidistribution} and \eqref{equality for the essmin in the proof of equidistribution} and simplify the terms (the weight $n_{w_0}$ is a positive real number) to obtain that
\begin{multline*}
\liminf_m\frac{t}{\deg_L(Z_m)}\int_{X_{v_0}^{\an}}f\ \c_1(\ol{L}_{v_0})^d\wedge\delta_{Z_{m,v_0}^{\an}}+O(t^2)
\\
\geq\frac{t}{\deg_L(X)}\int_{X^{\an}_{v_0}}f\ \c_1(\overline{L}_{v_0})^{\wedge\dim(X)}+O(t^2).
\end{multline*}

The previous inequality holds for every rational $t$ sufficiently small in absolute value.
In particular, for $t\longrightarrow0^+$ one has
\[
\liminf_m\frac{1}{\deg_L(Z_m)}\int_{X_{v_0}^{\an}}f\ \c_1(\ol{L}_{v_0})^d\wedge\delta_{Z_{m,v_0}^{\an}}\geq\frac{1}{\deg_L(X)}\int_{X^{\an}_{v_0}}f\ \c_1(\overline{L}_{v_0})^{\wedge\dim(X)},
\]
while for $t\longrightarrow0^-$
\[
\limsup_m\frac{1}{\deg_L(Z_m)}\int_{X_{v_0}^{\an}}f\ \c_1(\ol{L}_{v_0})^d\wedge\delta_{Z_{m,v_0}^{\an}}\leq\frac{1}{\deg_L(X)}\int_{X^{\an}_{v_0}}f\ \c_1(\overline{L}_{v_0})^{\wedge\dim(X)}.
\]
Comparing the two inequalities, one deduces that the net converges and that the limit coincides with the claimed one.

\end{proof}

As a special case, we obtain \hyperref[introthm:equidistribution]{Theorem~\ref*{introthm:equidistribution}} in the introduction.

\begin{proof}[Proof of Theorem~\ref*{introthm:equidistribution}]
Let $(Y_m)_m$ be a generic and $\ol{L}$-small net of subvarieties of $X_{\ol{K}}$ of dimension~$d$.
The net of Galois cycles of $Y_m$ is a generic and $\ol{L}$-small net of Galois invariant effective $d$-cycles of $X_{\ol{K}}$ because of \hyperref[rmk:genericity for Galois cycle]{Remark~\ref*{rmk:genericity for Galois cycle}} and \hyperref[rmk:smallness for Galois cycle]{Remark~\ref*{rmk:smallness for Galois cycle}}.
Therefore, the claim follows readily from \hyperref[equidistribution theorem]{Theorem~\ref*{equidistribution theorem}}.
\end{proof}

Another consequence is an equidistribution statement over smaller Berkovich spaces.

\begin{remark}
Let $(Z_m)_m$ be a $L$-generic and $\ol{L}$-small net of effective $d$-cycles of $X$.
By considering push-forwards by any of the natural maps $X_v^{\an}\to X_{K_v}^{\an}$, \hyperref[equidistribution theorem]{Theorem~\ref*{equidistribution theorem}} implies the weak convergence of the associated probability measures on~$X_{K_v}^{\an}$.
This is the higher dimensional version of~\cite[Theorem~3.2]{Yuan:2008}.
\end{remark}

We conclude the section with a question regarding a stronger (mixed) version of the equidistribution theorem for higher dimensional cycles.

\begin{question}
Let $d=0,\dots,\dim(X)$, and $\ol{L}$ be an ample line bundle on $X$ equipped with a semipositive quasi-algebraic metric.
Consider $(Z_m)_m$ an $L$-generic and $\ol{L}$-small net of Galois invariant effective cycles of $X_{\ol{K}}$ of pure dimension~$d$.

Is it true that for any choice of semipositive quasi-algebraic metrized ample line bundles $\ol{L}_1,\dots,\ol{L}_d$ the weak convergence of probability measures on $X_v^{\an}$
\begin{multline*}
\frac{1}{\deg_{L_1,\dots,L_d}(Y_m)}\ \c_1(\overline{L}_{1,v})\wedge\ldots\wedge\c_1(\ol{L}_{d,v})\wedge \delta_{Z_{m,v}^{\an}}
\longrightarrow
\\\frac{1}{\deg_{L_1,\dots,L_d,L,\dots,L}(X)}\ \c_1(\overline{L}_{1,v})\wedge\ldots\wedge\c_1(\ol{L}_{d,v})\wedge\c_1(\ol{L})^{\wedge (\dim(X)-d)}
\end{multline*}
 holds for every place~$v$?
\end{question}

%
%
%
%
%
%
%
%
%
%
%
%
%
%
%
%
%
%
%
%SECTION FURTHER COMMENTS
%
%
%
%
%
%
%
%
%
%
%
%
%
%
%
%
%
%
%

\section{Further comments}\label{section 5}

Throughout this section let $X$ be a projective variety defined over~$K$, and $\ol{L}$ a semipositive quasi-algebraic metrized line bundle on~$X$.

\subsection{Comparing with the literature}\label{subsection comparison}
Of course, Theorem~\ref{equidistribution theorem} is equivalent, when studying the equidistribution of small points, to the classical results in this area, and does not convey any new information, as the integral appearing in the definition of smallness is just the height of points.
The fundamental divisive element is the treatment of higher dimensional subvarieties.

In this subsection we compare Theorem~\ref{equidistribution theorem} to other equidistribution theorems for positive dimensional varieties present in literature.
In particular, we refer to the work of Autissier~\cite{Autissier:2006}, Yuan~\cite{Yuan:2008} and, implicitly, Gubler~\cite{Gubler:2008}.

The main difference between the equidistribution in Theorem~\ref{equidistribution theorem} and the aforementioned ones is in the hypothesis of the statements.
Precisely, the notion of a small sequence of $d$-dimensional subvarieties $(Y_m)_m$ of $X$ in~\cite{Autissier:2006,Yuan:2008,Gubler:2008} is defined using the convergence of their normalized heights; that is,
\begin{equation}\label{eq:alternative-small}
\lim_{m\rightarrow \infty}\widehat{\h}_{\ol{L}}(Y_m)\;=\;\widehat{\h}_{\ol{L}}(X).
\end{equation}
Compare this to Definition~\ref{definition smallness}, even under the assumption that $\e_1^{(d)}(X,\ol{L})=\h_{\ol{L}}(X)$.
This ``seemingly'' simplification of the notion of \emph{small} comes however with the following required extra hypothesis on the metrized line bundle so that the respective equidistribution theorems in \emph{loc.~cit.} hold.
\begin{assumption}\label{eq:condition}
Fixed $d>0$,  
\begin{equation*}
\widehat{\h}_{\ol{L}}(Y)\;\geq\;\widehat{\h}_{\ol{L}}(X),
\end{equation*}
for every subvariety $Y$ of $X_{\ol{K}}$ of dimension $(d-1)$.
\end{assumption}

Notice that this assumption implies also that every effective $(d-1)$-cycle $Z$ of $X_{\ol{K}}$ satisfies the same inequality.
This hypothesis is verified in a large number of cases which are of intrinsic interest. For instance, canonical metrics of toric line bundles, N\'{e}ron-Tate height of abelian varieties, and, more generally, canonical heights associated to dynamical systems, see~Subsection~5.B.
Nevertheless, we can easily produce an example where Assumption~\ref{eq:condition} does not hold.

\begin{example}\label{ex:failure of hypothesis *}
Let $X=\PP^n_\QQ$ be the projective space over $\QQ$, and $L=\mathscr{O}(1)$ together with the Fubini-Study metric.
For a point $p=(p_0:\cdots:p_n)\in\PP^n(\ol{\QQ})$ and a regular section $s$ of $\mathscr{O}(1)$, which we identify with a homogeneous lineal polynomial~$f$,
these metrics are given by
\[
\norm{s(p)}_v=\left\lbrace
\begin{array}{ll}
\frac{\abs{f(p_0,\ldots,p_n)}_v}{(\sum_{i}\abs{p_i}_v^2)^{1/2}},
&\mbox{if }v\mid \infty;\vspace{4mm}\\
\frac{\abs{f(p_0,\ldots,p_n)}_v}{\max_{i}(\abs{p_i}_v)},
&\mbox{if }v\nmid\infty.
\end{array}\right.
\]
For short, we denote this metrized line bundle by $\ol{\mathscr{O}(1)}^{\mathrm{FS}}$.

One can compute explicitly the height of $X$ with respect to this metric
\[
\h_{\ol{\scrO(1)}^{\mathrm{FS}}}(\PP_\QQ^n)= \frac{n+1}{2}\sum_{j=2}^{n+1}\frac{1}{j},
\]
see for instance~\cite[Lemma~3.3.1]{Bost-Gillet-Soule}. %or~\cite[Example~6.2.6]{Burgos-Philippon-Sombra:2014}.

The minimal height of points with respect to $\ol{\mathscr{O}(1)}^{\mathrm{FS}}$ is $0$ (\cite[Théorème~0.1]{Sombra:msvtp}) which already contradicts Assumption~\ref{eq:condition} for $d=1$.
Further immediate examples can be given, when $n=3$. Let $\bfzeta=(\zeta_1,\zeta_2,\zeta_3)\in\Gm^3(\ol{\QQ})$ be a torsion point, and denote by $C_{\bfzeta}\subseteq\PP^3_{\ol{\QQ}}$ the translate of the Veronese curve of degree~$3$ by~$\bfzeta$, that is the closure of the image of the morphism
\[
\Gm\longrightarrow\PP^3,\quad t\longmapsto (1:\zeta_1 t:\zeta_2 t^2:\zeta_3 t^3).
\]
Then \cite[Corollary~7.1.6]{Burgos-Philippon-Sombra:2014} and the invariance of the Fubini-Study metric under torsion translates give
\[
\h_{\ol{\scrO(1)}^{\mathrm{FS}}}(C_{\bfzeta})=\frac{3}{2}+\pi\bigg(1-\frac{2}{4}\bigg)\cot\bigg(\frac{\pi}{4}\bigg).
\]
Hence this also contradicts Assumption~\ref{eq:condition} for $d=2$, as
\[
\widehat{\h}_{\ol{\scrO(1)}^{\mathrm{FS}}}(C_{\bfzeta})=
\frac{\h_{\ol{\scrO(1)}^{\mathrm{FS}}}(C_{\bfzeta})}{2\deg(C_{\bfzeta})}=\frac{1}{4}+\frac{\pi}{12}
<
\frac{13}{24}=\frac{\h_{\ol{\scrO(1)}^{\mathrm{FS}}}(\PP^3)}{4}
=\widehat{\h}_{\ol{\scrO(1)}^{\mathrm{FS}}}(\PP^3).
\]
\end{example}

\vspace{\baselineskip}

Even if Assumption~\ref{eq:condition} does not hold in general, we restrict to its setting to compare both notions of smallness, meaning Definition~\ref{definition smallness} and~\eqref{eq:alternative-small}.
To do so, we first give the following lemma, which may be related to Corollary~\ref{cor:Zhangs-ineq} and serves already as a first comparison point between both contexts.
For simplicity, we assume also that $L$ is ample, although this hypothesis may be omitted by a careful use of genericity (using Proposition~\ref{prop:generic-big}).
\begin{lemma}\label{lemma:Zhang-normalized-height}
Let $\ol{L}$ be a semipositive metrized ample line bundle satisfying Assumption~\ref{eq:condition} for a fixed~$d=1,\ldots,\dim(X)$. Then
\begin{equation}\label{eq:Zhang-normalized-height}
\sup_{\substack{H\text{ closed subset of }X_{\ol{K}}\\ \codim(H)=1}}\;\inf_{\substack{Y\subseteq X_{\ol{K}}\\ \dim(Y)=d\\ Y\nsubseteq H}}\widehat{\h}_{\ol{L}}(Y)\geq \widehat{\h}_{\ol{L}}(X).
\end{equation}
\end{lemma}
\begin{proof}
Combining Corollary~\ref{cor:Zhangs-ineq}, Proposition~\ref{prop:equivalent-def-ess-minimum}, Remark~\ref{correcting integrals as height jumps} and Assumption~\ref{eq:condition}, we have that
\begin{multline*}
\widehat{\h}_{\ol{L}}(X)\leq \sup_{\substack{H\text{\mbox{ closed subset of }}X_{\ol{K}}\\\codim(H)=1}}\;
\inf_{\substack{Y\subseteq X_{\ol{K}}\\ \dim(Y)=d\\ Y\nsubseteq H}}
\sup_{\substack{s\in\mathrm{H}^0(X_{\ol{K}},L_{\ol{K}}^{\otimes n})\\n\in\mathbb{N}\setminus\{0\}\\Y_m\nsubseteq|\div(s)|}}
\bigg((d+1)\widehat{\h}_{\ol{L}}(Y)-d\;\widehat{\h}_{\ol{L}}(\div(s)\cdot Y)\bigg)\\
\leq \sup_{\substack{H\text{\mbox{ closed subset of }}X_{\ol{K}}\\ \codim(H)=1}}\;\inf_{\substack{Y\subseteq X_{\ol{K}}\\ \dim(Y)=d\\ Y\nsubseteq H}}(d+1)\widehat{\h}_{\ol{L}}(Y)-d\;\widehat{\h}_{\ol{L}}(X),
\end{multline*}
from which we readily deduce the statement.
\end{proof}

Notice that this lemma and the notion of smallness in~\eqref{eq:alternative-small} motivate the definition of an alternative version of $d$-dimensional essential minimum as the value on the left-hand side of~\eqref{eq:Zhang-normalized-height}.
However, the inequality in this lemma does not hold in general. %without Hypothesis~\ref{eq:condition}
\begin{example}\label{ex:failure of Zhang ineq}
Following the same notation as in Example~\ref{ex:failure of hypothesis *},
the family $(C_{\bfzeta})_{\bfzeta}$, where $\bfzeta$ ranges over all torsion points in $\Gm^3(\ol{\QQ})$, is generic. %\bfzeta\in C_{\bfzeta}$
Therefore
\[
\sup_{\substack{H\text{ closed subset of }\PP^3_{\ol{\QQ}}\\\codim(H)=1}}\;\inf_{\substack{Y\subseteq \PP^3_{\ol{\QQ}}\\ \dim(Y)=1\\ Y\nsubseteq H}}\widehat{\h}_{\ol{\scrO(1)}^{\mathrm{FS}}}(Y)
\leq \frac{1}{4}+\frac{\pi}{12}<\widehat{\h}_{\ol{\scrO(1)}^{\mathrm{FS}}}(\PP^3).
\]
\end{example}

To further display the difference between both notions of small generic sequences we present the following result.

\begin{proposition}\label{prop:autissier}
Let $L$ be an ample line bundle on~$X$.
Fix $d= 1,\ldots, \dim(X)$, and let $\ol{L}$ be $L$ together with a semipositive quasi-algebraic metric such that Assumption~\ref{eq:condition} is satisfied for~$d$.
Let $(Y_m)_m$ be a generic net of $d$-dimensional subvarieties of~$X_{\ol{K}}$.
If equation~\eqref{eq:alternative-small} is satisfied for~$(Y_m)_m$, then $(Y_m)_m$ is $\ol{L}$-small (as in Definition~\ref{definition smallness})
and moreover  $\e_1^{(d)}(X,\ol{L})=\widehat{\h}_{\ol{L}}(X)$.
\end{proposition}

\begin{proof}
Assuming $\widehat{\h}_{\ol{L}}(Y_m)$ converges to $\widehat{\h}_{\ol{L}}(X)$,
write
\begin{multline*}
L_{sup}:=\limsup_m\sup_{\substack{s\in\mathrm{H}^0(X_{\ol{K}},L^{\otimes n}_{\ol{K}})\\ n\in\NN\setminus\lbrace0\rbrace\\ Y_m\nsubseteq |\div(s)|}}
\frac{I_{\ol{L}}(Y_m,s)}{\deg_L(Y_m)}\\
=
\limsup_m\sup_{\substack{s\in\mathrm{H}^0(X_{\ol{K}},L^{\otimes n}_{\ol{K}})\\ n\in\NN\setminus\lbrace0\rbrace\\ Y_m\nsubseteq |\div(s)|}}
\big((d+1)\widehat{\h}_{\ol{L}}(Y_m)-d\;\widehat{\h}_{\ol{L}}(\div(s)\cdot Y_m)\big),
\end{multline*}
where the second equality is due to Remark~\ref{correcting integrals as height jumps}.
Since the global height of the $Y_m$'s is independent on the choice of the section, we further get
\[
L_{sup}=\limsup_{m}\bigg((d+1)\;\widehat{\h}_{\ol{L}}(Y_m)-d \inf_{\substack{s\in\mathrm{H}^0(X_{\ol{K}},L^{\otimes n}_{\ol{K}})\\ n\in\NN\setminus\lbrace0\rbrace\\ Y_m\nsubseteq |\div(s)|}}\widehat{\h}_{\ol{L}}(\div(s)\cdot Y_m)\bigg).
\]
By Assumption~\ref{eq:condition}, the interior of the $\limsup$ is bounded above by $(d+1)\;\widehat{\h}_{\ol{L}}(Y_m)-d\;\widehat{\h}_{\ol{L}}(X)$ for every~$m$.
Since $\lim(d+1)\;\widehat{\h}_{\ol{L}}(Y_m)=(d+1)\;\widehat{\h}_{\ol{L}}(X)$ by hypothesis,
we conclude that $L_{\sup}\leq \widehat{\h}_{\ol{L}}(X)$.

On the other hand, let $L_{inf}$ be defined equivalently to $L_{sup}$, replacing the limit superior by a limit inferior. 
Then, since $(Y_m)_m$ is generic, by Proposition~\ref{prop:inequality liminf-essential minimum} and Corollary~\ref{cor:Zhangs-ineq} we have that $L_{inf}\geq\e_1^{(d)}(X,\ol{L})\geq \widehat{\h}_{\ol{L}}(X)$.

Therefore $L_{sup}=L_{inf}=\widehat{\h}_{\ol{L}}(X)=\e_1^{(d)}(X,\ol{L})$, which concludes the proof.
\end{proof}

This proposition illustrates the fact that Theorem~\ref{introthm:equidistribution} contains the equidistribution theorems for positive dimensional varieties present in literature, in particular~\cite{Autissier:2006,Yuan:2008}.

\subsection{Dynamical heights}
Let $L$ be ample.
Let $f\colon X\rightarrow X$ be a surjective morphism such that $L^{\otimes d}\simeq f^*L$ for some integer $d\geq 2$.
We call the triple $(X,f,L)$ an \emph{algebraic dynamical system}.

A result of Zhang~\cite[Theorem~2.2]{Zhang:adelic-metrics} gives the construction of a \emph{canonical metric} associated to $(X,f,L)$.
Fixed an isomorphism $\varphi\colon L^{\otimes d}\xrightarrow{\simeq} f^*L$, there exists a unique semipositive quasi-algebraic metric on $L$
such that $\varphi$ determines an isometry between $\ol{L}^{\otimes d}$ and $f^*\ol{L}$.
When $X$ is a polarized toric variety, choosing $f$ as the extension of the morphism
\[
\Gm^n\longrightarrow \Gm^n,\quad t\rightarrow t^k,
\]
for any choice $k\in\NN$, corresponds to the canonical metric associated to $X$~\cite[\S3.4]{Maillot:GAdvt};
we also refer to~\cite[Proposition-Definition~4.3.15]{Burgos-Philippon-Sombra:2014} for the definition of the canonical metric on~$X$.
On the other hand, when $X$ is an abelian variety and $f$ the multiplication-by-$n$ endomorphism, $n\in\NN_{>1}$, we obtain the Néron-Tate metric on $L$~\cite[\S3]{Zhang:adelic-metrics}.

In the case of such dynamical metrics, every subvariety $Y$ of $X_{\ol{K}}$ has nonnegative height.
Moreover, if $Y$ is a \emph{preperiodic subvariety}, that is $\lbrace f^m(Y),m\in\NN\rbrace$ is finite, then it has height~$0$.
In particular $\mathrm{h}_{\ol{L}}(X)=0$, which automatically guarantees that Assumption~\ref{eq:condition} is always satisfied in the case of dynamical heights.

The dynamical version of Theorem~\ref{equidistribution theorem} amounts to the following.
\begin{theorem}
Let $(X,f,L)$ be an algebraic dynamical system, defining the semipositive quasi-algebraic metrized line bundle~$\ol{L}$.
Let $d=0,\ldots,\dim(X)$, and $(Z_m)_{m}$ be an $L$-generic net of Galois invariant effective $d$-cycles of~$X_{\ol{K}}$.
Write $Z_m=\sum_{i}a_{m,i}\,Y_{m,i}$.
If
\[
\lim_m\,\frac{1}{\deg_L(Z_m)}\sum_ia_{m,i}\sup_{\substack{s\in\mathrm{H}^0(X_{\ol{K}},L_{\ol{K}}^{\otimes n})\\n\in\mathbb{N}\setminus\{0\}\\Y_{m,i}\nsubseteq|\div(s)|}}I_{\ol{L}}(Y_{m,i},s)
=0,
\]
then the weak convergence of measures on~$X_v^{\an}$
\[
\frac{1}{\deg_L(Z_m)}\ \c_1(\ol{L}_v)^{\wedge d}\wedge\delta_{Z_{m,v}^{\an}}
\longrightarrow
\frac{1}{\deg_L(X)}\ \c_1(\ol{L}_v)^{\wedge \dim(X)}
\]
holds for every place~$v\in\fM$.
\end{theorem}

In this setting, it is easy to give sufficient conditions for which the hypotheses of the equidistribution theorem are satisfied, namely $\e_1^{(d)}(X,\ol{L})=\widehat{\h}_{\ol{L}}(X)=0$.

\begin{proposition}\label{prop:suf-cond}
Let $(X,f,L)$ be an algebraic dynamical system, $\ol{L}$ be its associated canonical height, and $d =0,\dots,\dim(X)$.
Assume that the  $d$-dimensional subvarieties of $X_{\ol{K}}$ having height equal to $0$ are dense in~$X_{\ol{K}}$.
Then $\e^{(d)}_1(X,\overline{L})=0$.

In particular, the equality holds if $(X,\ol{L})$ is a polarized toric variety endowed with the canonical metric, or a polarized abelian variety together with the associated Néron-Tate metric.
\end{proposition}
\begin{proof}
By Corollary~\ref{cor:Zhangs-ineq}, we have that $\e_1^{(d)}(X,\ol{L})\geq 0$.
On the other hand, for every closed subset $H\subseteq X_{\ol{K}}$ of codimension~$1$, we can find a $d$-dimensional subvariety $Y_H$ of height~$0$ such that $Y_H\nsubseteq H$.
Therefore, by Proposition~\ref{prop:equivalent-def-ess-minimum},
\[
\e_1^{(d)}(X,\ol{L})\leq \sup_H \sup_s \big((d+1)\widehat{\h}_{\ol{L}}(Y_H)-d\widehat{\h}_{\ol{L}}(\div(s)\cdot Y_H)\big)
\leq 0,
\]
where the second inequality follows from the fact that $\h_{\ol{L}}(Y_H)=0$ and $\h_{\ol{L}}(Z)\geq 0$ for every effective cycle on~$X$.
This concludes the proof.
\end{proof}

\begin{remark}
The case of semiabelian varieties is of particular interest.
By taking a canonical height associated to a fixed semiabelian variety~$X$,
that is, associated to a semipositive quasi-algebraic metrized line bundle~$\ol{L}$ on $X$ as for instance in~\cite[\S4]{Chambert-Loir:1999},
one sees that $\e_1^{(0)}(X,\ol{L})=0$ by the proposition above.

The study of higher dimensional essential minima is far more difficult, starting by the fact that already in the case of an abelian variety one cannot expect there to be algebraic subgroups of every dimension (to then apply Proposition~\ref{prop:suf-cond}).
Let $(Y_m)_m$ be a generic net of $d$-dimensional subvarieties of $X_{\ol{K}}$.

In the case when $X$ is split (that is, isogenous to a product of a torus and an abelian variety), we have that $\widehat{\h}_{\ol{L}}(X)=0$.
Hence, if
\[
\sup_s \frac{I(Y_m,s)}{\deg_L(Y_m)}\longrightarrow 0,
\]
then $(Y_m)_m$ equidistributes in the sense of Theorem~\ref{introthm:equidistribution}.

In the case when $X$ is not split, the techniques developed by K\"uhne in~\cite{Kuhne:2018} are specially helpful for determining expected equidistribution.
The idea is to look at the whole isogeny class of $X$ instead of merely $X$ itself.
Following the discussion in \S3 of \emph{loc.~cit.}, one can choose a sequence of pairs $(X_n,\ol{L}_n)$ such that $X_n$ is isogenous to $X$ and $\ol{L}_n$ defines a canonical metric on $X_n$ such that $\h_{\ol{L}_n}(X_n)\rightarrow 0$.
In particular, if we denote by $Y_{m,n}$ the image of $Y_m$ in $X_n$,
\[
\lim_{m,n}\sup_s \frac{I(Y_{m,n},s)}{\deg_L(Y_{m,n})}\longrightarrow 0
\]
is a sufficient condition for the net $(Y_m)_m$ to equidistribute in the sense of Theorem~\ref{introthm:equidistribution},
giving an example of higher dimensional equidistribution without assuming necessarily $e_1^{(d)}(X,\ol{L})=\widehat{h}_{\ol{L}}(X)$.
\end{remark}

%
%
%
%
%
%
%
%
%
%
%
%
%
%
%
%
%
%
%
%SECTION 5
%
%
%
%
%
%
%
%
%
%
%
%
%
%
%
%
%
%
%

\bibliographystyle{amsalpha}
\bibliography{biblio}

\providecommand{\bysame}{\leavevmode\hbox to3em{\hrulefill}\thinspace}
\providecommand{\MR}{\relax\ifhmode\unskip\space\fi MR }
\providecommand{\MRhref}[2]{%
  \href{http://www.ams.org/mathscinet-getitem?mr=#1}{#2}
}
\providecommand{\href}[2]{#2}
\begin{thebibliography}{{Cha}19}

\bibitem[Aut06]{Autissier:2006}
P.~Autissier, \emph{\'{E}quidistribution des sous-vari\'et\'es de petite
  hauteur}, J. Th\'eor. Nombres Bordeaux \textbf{18} (2006).

\bibitem[Ber90]{Berkovich:1990}
V.~G. Berkovich, \emph{Spectral theory and analytic geometry over
  non-{A}rchimedean fields}, Mathematical Surveys and Monographs, vol.~33,
  American Mathematical Society, Providence, RI, 1990.

\bibitem[BG06]{Bombieri-Gubler:2006}
E.~Bombieri and W.~Gubler, \emph{{Heights in {D}iophantine geometry}}, {New
  Mathematical Monographs}, Cambridge University Press, Cambridge, 2006.

\bibitem[BGS94]{Bost-Gillet-Soule}
J.-B. Bost, H.~Gillet, and C.~Soul\'e, \emph{Heights of projective varieties
  and positive {G}reen forms}, J. Amer. Math. Soc. \textbf{7} (1994), no.~4,
  903--1027.

\bibitem[BI04]{Baker-Ih:equidistribution}
M.~Baker and S.~Ih, \emph{Equidistribution of small subvarieties of an abelian
  variety}, New York J. Math. \textbf{10} (2004), 279--285.

\bibitem[Bil97]{Bilu:1997}
Yu. Bilu, \emph{Limit distribution of small points on algebraic tori}, Duke
  Math. J. \textbf{89} (1997), no.~3, 465--476. \MR{1470340}

\bibitem[BPRS19]{Burgos-Philippon-Rivera-Sombra:2015}
J.~I. {Burgos Gil}, P.~{Philippon}, J.~{Rivera-Letelier}, and M.~{Sombra},
  \emph{The distribution of {G}alois orbits of points of small height in toric
  varieties}, Amer. J. Math. \textbf{141} (2019), no.~2, 309--381.

\bibitem[BPS14]{Burgos-Philippon-Sombra:2014}
J.~I. {Burgos Gil}, P.~Philippon, and M.~Sombra, \emph{Arithmetic geometry of
  toric varieties. {M}etrics, measures and heights}, Ast\'erisque, vol. 360,
  Soc. Math. France, 2014.

\bibitem[{Cha}99]{Chambert-Loir:1999}
A.~{Chambert-Loir}, \emph{G\'{e}om\'{e}trie d'{A}rakelov et hauteurs canoniques
  sur des vari\'{e}t\'{e}s semi-ab\'{e}liennes}, Math. Ann. \textbf{314}
  (1999), no.~2, 381--401.

\bibitem[{Cha}06]{Chambert-Loir:mesures-equidistribution}
\bysame, \emph{Mesures et \'equidistribution sur les espaces de {B}erkovich},
  J. Reine Angew. Math. \textbf{595} (2006), 215--235.

\bibitem[{Cha}11]{Chambert-Loir:2011}
\bysame, \emph{Heights and measures on analytic spaces. {A} survey of recent
  results, and some remarks}, Motivic integration and its interactions with
  model theory and non-{A}rchimedean geometry. {V}olume {II}, London Math. Soc.
  Lecture Note Ser., vol. 384, Cambridge Univ. Press, Cambridge, 2011,
  pp.~1--50.

\bibitem[{Cha}19]{Chambert-Loir:Bogomolov-survey}
\bysame, \emph{{Arakelov geometry, heights, equidistribution, and the Bogomolov
  conjecture}}, arXiv e-prints (2019), arXiv:1904.05630.

\bibitem[{Che}08]{Chen:Positive-degree}
H.~{Chen}, \emph{{Positive degree and arithmetic bigness}}, arXiv e-prints
  (2008), arXiv:0803.2583.

\bibitem[CT09]{Chambert-Loir-Thuillier:equidistribution-logarithmique}
A.~{Chambert-Loir} and A.~Thuillier, \emph{Mesures de {M}ahler et
  \'equidistribution logarithmique}, Ann. Inst. Fourier (Grenoble) \textbf{59}
  (2009), no.~3, 977--1014.

\bibitem[Dem93]{Demailly:LelongNumbers}
J.~P. Demailly, \emph{Monge-{A}mp\`ere operators, {L}elong numbers and
  intersection theory}, Complex analysis and geometry, Univ. Ser. Math.,
  Plenum, New York, 1993, pp.~115--193.

\bibitem[Duj17]{Dujardin}
R.~Dujardin, \emph{Some problems of arithmetic origin in rational dynamics},
  Notes of a summer course at Grenoble (2017),
  \url{https://www.lpsm.paris/pageperso/dujardin/latex/notes_grenoble.pdf}.

\bibitem[FR06]{Favre-Rivera:2006}
C.~{Favre} and J.~{Rivera-Letelier}, \emph{\'{E}quidistribution quantitative
  des points de petite hauteur sur la droite projective}, Math. Ann.
  \textbf{335} (2006), no.~2, 311--361.

\bibitem[GH17]{Gubler-Hertel}
W.~Gubler and J.~Hertel, \emph{Local heights of toric varieties over
  non-archimedean fields}, Actes de la {C}onf\'{e}rence ``{N}on-{A}rchimedean
  {A}nalytic {G}eometry: {T}heory and {P}ractice'', Publ. Math. Besancon
  Alg\`ebre Th\'{e}orie Nr., vol. 2017/1, Presses Univ. Franche-Comt\'{e},
  Besancon, 2017, pp.~5--77. \MR{3752487}

\bibitem[GK15]{Gubler-Kuennemann:positivity}
W.~Gubler and K.~K\"{u}nnemann, \emph{{Positivity properties of metrics and
  delta-forms}}, arXiv e-prints (2015), arXiv:1509.09079.

\bibitem[GK17]{Gubler-Kuennemann:2017}
\bysame, \emph{A tropical approach to nonarchimedean {A}rakelov geometry},
  Algebra Number Theory \textbf{11} (2017), no.~1, 77--180. \MR{3602767}

\bibitem[GM19]{Gubler-Martin}
W.~Gubler and F.~Martin, \emph{On {Z}hang's semipositive metrics}, Doc. Math.
  \textbf{24} (2019), 331--372.

\bibitem[Gub97]{Gubler:Mfields}
W.~Gubler, \emph{Heights of subvarieties over {$M$}-fields}, Arithmetic
  geometry ({C}ortona, 1994), Sympos. Math., XXXVII, Cambridge Univ. Press,
  1997, pp.~190--227.

\bibitem[Gub98]{Gubler:localheights}
\bysame, \emph{Local heights of subvarieties over non-{A}rchimedean fields}, J.
  Reine Angew. Math. \textbf{498} (1998), 61--113.

\bibitem[Gub07]{Gubler:tropical}
\bysame, \emph{Tropical varieties for non-{A}rchimedean analytic spaces},
  Invent. Math. \textbf{169} (2007), no.~2, 321--376.

\bibitem[Gub08]{Gubler:2008}
\bysame, \emph{Equidistribution over function fields}, Manuscripta Math.
  \textbf{127} (2008), 485--510.

\bibitem[K{\"u}h18]{Kuhne:2018}
L.~K{\"u}hne, \emph{{Points of Small Height on Semiabelian Varieties}}, ArXiv
  e-prints (2018).

\bibitem[Laz04]{Lazarsfeld:positivityI}
R.~Lazarsfeld, \emph{Positivity in algebraic geometry. {I}}, Ergeb. Math.
  Grenzgeb. (3), vol.~48, Springer-Verlag, 2004.

\bibitem[Mai00]{Maillot:GAdvt}
V.~Maillot, \emph{G\'eom\'etrie d'{A}rakelov des vari\'et\'es toriques et
  fibr\'es en droites int\'egrables}, M\'em. Soc. Math. France, vol.~80, Soc.
  Math. France, 2000.

\bibitem[Mor00]{Moriwaki:Finitely-generated}
A.~Moriwaki, \emph{Arithmetic height functions over finitely generated fields},
  Invent. Math. \textbf{140} (2000), no.~1, 101--142. \MR{1779799}

\bibitem[Mor09a]{Moriwaki:Continuity}
\bysame, \emph{Continuity of volumes on arithmetic varieties}, J. Algebraic
  Geom. \textbf{18} (2009), no.~3, 407--457. \MR{2496453}

\bibitem[Mor09b]{Moriwaki:Continuity2}
\bysame, \emph{Continuous extension of arithmetic volumes}, Int. Math. Res.
  Not. IMRN (2009), no.~19, 3598--3638. \MR{2539186}

\bibitem[Mor14]{Moriwaki:Arakelov}
\bysame, \emph{Arakelov geometry}, Translations of Mathematical Monographs,
  vol. 244, American Mathematical Society, Providence, RI, 2014.

\bibitem[MS19]{Martinez-Sombra:BKK}
C.~Mart\'{\i}nez and M.~Sombra, \emph{An arithmetic
  {B}ern\v{s}tein-{K}u\v{s}nirenko inequality}, Math. Z. \textbf{291} (2019),
  no.~3-4, 1211--1244.

\bibitem[Ser79]{Serre:local-fields}
J-P. Serre, \emph{Local fields}, Graduate Texts in Mathematics, vol.~67,
  Springer-Verlag, New York-Berlin, 1979, Translated from the French by Marvin
  Jay Greenberg. \MR{554237}

\bibitem[Som05]{Sombra:msvtp}
M.~Sombra, \emph{Minimums successifs des vari\'et\'es toriques projectives}, J.
  Reine Angew. Math. \textbf{586} (2005), 207--233.

\bibitem[SUZ97]{SUZ}
L.~Szpiro, E.~Ullmo, and S.~Zhang, \emph{\'{E}quir\'{e}partition des petits
  points}, Invent. Math. \textbf{127} (1997), no.~2, 337--347. \MR{1427622}

\bibitem[Ull98]{Ullmo:Bogomolov}
E.~Ullmo, \emph{Positivit\'{e} et discr\'{e}tion des points alg\'{e}briques des
  courbes}, Ann. of Math. (2) \textbf{147} (1998), no.~1, 167--179.
  \MR{1609514}

\bibitem[Yam17]{Yamaki:survey}
K.~Yamaki, \emph{Survey on the geometric {B}ogomolov conjecture}, Actes de la
  {C}onf\'{e}rence ``{N}on-{A}rchimedean {A}nalytic {G}eometry: {T}heory and
  {P}ractice'', Publ. Math. Besan{\cedille}on Alg\`ebre Th\'{e}orie Nr., vol.
  2017/1, Presses Univ. Franche-Comt\'{e}, Besan{\cedille}on, 2017,
  pp.~137--193.

\bibitem[Yua08]{Yuan:2008}
X.~Yuan, \emph{Big line bundles over arithmetic varieties}, Invent. Math.
  \textbf{173} (2008), 603--649.

\bibitem[Zha95a]{Zhang:1995}
S.~Zhang, \emph{{Positive line bundles on arithmetic varieties}}, J. Amer.
  Math. Soc. \textbf{8} (1995), 187--221.

\bibitem[Zha95b]{Zhang:adelic-metrics}
\bysame, \emph{Small points and adelic metrics}, J. Algebraic Geom. \textbf{4}
  (1995), no.~2, 281--300.

\bibitem[Zha98]{Zhang:Bogomolov}
\bysame, \emph{Equidistribution of small points on abelian varieties}, Ann. of
  Math. (2) \textbf{147} (1998), no.~1, 159--165. \MR{1609518}

\end{thebibliography}

\end{document}